\documentclass[10pt]{article}
\usepackage{bbm}
\usepackage{amsfonts}
\usepackage{amsmath}
\allowdisplaybreaks
\usepackage{amssymb}
\usepackage{fancyhdr}
\usepackage{setspace}
\usepackage{latexsym}
\usepackage{mathrsfs}
\usepackage{amsthm}
\usepackage{floatrow}
\usepackage{pdfsync}
\usepackage{tikz}
\usepackage{subfigure}

\usepackage{cite,enumitem,graphicx}
\usepackage[backref,colorlinks,linkcolor=red,anchorcolor=green,citecolor=blue]{hyperref}
\usepackage{cleveref}
\usepackage[english]{babel}

\numberwithin{equation}{section}
\newtheorem{theorem}{Theorem}[section]
\newtheorem{proposition}[theorem]{Proposition}
\newtheorem{lemma}[theorem]{Lemma}
\newtheorem{corollary}[theorem]{Corollary}
\newtheorem{definition}[theorem]{Definition}
\newtheorem{remark}[theorem]{Remark}
\newtheorem{example}[theorem]{Example}
\newcommand{\ud}{\mathrm{d}}

\newcommand{\bt}{\begin{theorem}}
\newcommand{\et}{\end{theorem}}
\newcommand{\bl}{\begin{lemma}}
\newcommand{\el}{\end{lemma}}
\newcommand{\bd}{\begin{definition}}
\newcommand{\ed}{\end{definition}}
\newcommand{\bc}{\begin{corollary}}
\newcommand{\ec}{\end{corollary}}
\newcommand{\bp}{\begin{proof}}
\newcommand{\ep}{\end{proof}}
\newcommand{\bx}{\begin{example}}
\newcommand{\ex}{\end{example}}
\newcommand{\bi}{\begin{exercise}}
\newcommand{\ei}{\end{exercise}}
\newcommand{\bo}{\begin{proposition}}
\newcommand{\eo}{\end{proposition}}
\newcommand{\br}{\begin{remark}}
\newcommand{\er}{\end{remark}}
\newcommand{\be}{\begin{equation}}
\newcommand{\ee}{\end{equation}}
\newcommand{\ba}{\begin{align}}
\newcommand{\ea}{\end{align}}
\newcommand{\bn}{\begin{enumerate}}
\newcommand{\en}{\end{enumerate}}
\newcommand{\bg}{\begin{align*}}
\newcommand{\bcs}{\begin{cases}}
\newcommand{\ecs}{\end{cases}}

\newcommand{\bean}{\begin{eqnarray*}}
\newcommand{\eean}{\end{eqnarray*}}

\makeatletter % @ is now a normal "letter" for TeX
\renewcommand\theequation{\thesection.\arabic{equation}}
\@addtoreset{equation}{section}
\makeatother%@isrestoredasa"non-letter"characterforTeX
\numberwithin{equation}{section}

\begin{document}

\begin{center}
\textbf{Existence and nonexistence of normalized solutions for nonlinear Schr\"{o}dinger equation involving combined nonlinearities  in bounded domain}\\
\end{center}

\begin{center}
{ Zhen-Feng Jin$^{1}$, Weimin Zhang$^{2}$\\

$^{1}$ Shanxi Key Laboratory of Cryptography and Data Security \& School of Mathematical Sciences, Shanxi Normal University, Taiyuan, Shanxi, 030031, P.R. China\\
\smallskip
$^{2}$ School of Mathematical Sciences, Zhejiang Normal University, Jinhua, 321004, P.R. China
\smallskip
}

\end{center}
\begin{center}
\renewcommand{\theequation}{\arabic{section}.\arabic{equation}}
\numberwithin{equation}{section}
\footnote[0]{\hspace*{-7.4mm}
%$^*$Corresponding author.\\
{E-mail addresses: jinzhenfeng@sxnu.edu.cn (Z. Jin), zhangweimin2021@gmail.com (W. Zhang).}}
\footnote[0]{\hspace{-7.4mm}
{ Z. Jin was supported by Fundamental Research Program of Shanxi Province of China(No. 202303021212160) and the China Scholarship Council(No. 202508140048).}}
\end{center}

\begin{abstract}
\noindent In this paper, we consider the existence, multiplicity and nonexistence of  solutions for the following equation
\begin{equation*}
\begin{cases}
\begin{aligned}
&-\Delta u+\omega u=\mu u^{p-1}+u^{q-1},~ u>0 \quad &&\text { in } \Omega, \\
&u=0  &&\text { on } \partial\Omega, \\
\end{aligned}
\end{cases}
\end{equation*}
with prescribed $L^2$-norm $\|u\|_2^2=\rho$, where $N\ge 1$, $\rho>0$, $\mu\in \mathbb{R}$, $1<p\le q$, and $\Omega\subset\mathbb{R}^N$ is a bounded smooth domain. The parameter $\omega\in\mathbb{R}$ arises as a Lagrange multiplier. Firstly, when $2<p\le q\le \frac{2N}{(N-2)^+}$ and $\rho$ is small, we establish the existence of a local minimizer of energy. Furthermore, when $\mu\ge 0$ and $\Omega$ is a star-shaped domain, using the monotonicity trick and the Pohozaev identity, we show that there exists a second solution which is of mountain pass type. Secondly, when  $\mu\ge 0$, $N\ge 3$, $1<p\le 2$, $q\ge \max\left\{\frac{2N}{N-2}, 3\right\}$ and $\Omega$ is a convex domain, using the moving-plane method, we prove the nonexistence of normalized solutions for large $\rho$. Finally, when $\mu=0$, $N\ge 3$, $q=\frac{2N}{N-2}$ and $\Omega$ is a ball, we give a dichotomy result of normalized solutions for the Br\'{e}zis-Nirenberg problem by continuation arguments.
 \end{abstract}

\section{Introduction}\label{s1}
This paper is concerned with the existence and nonexistence of solutions for the following nonlinear Schr\"{o}dinger equation involving combined power nonlinearities \begin{equation}\label{041003}
\begin{cases}
\begin{aligned}
&-\Delta u+\omega u=\mu u^{p-1}+u^{q-1},~ u>0 \quad &&\text { in } \Omega, \\
&u=0  &&\text { on } \partial\Omega, \\
&\|u\|_2^2=\rho>0,
\end{aligned}
\end{cases}
\end{equation}
where $N\ge 1$, $\mu\in \mathbb{R}$, $1<p\le q$  and $\Omega\subset\mathbb{R}^N$ is a bounded smooth domain. Here the parameter $\omega\in\mathbb{R}$ is not fixed, which arises as a Lagrange multiplier with respect to the mass constraint $\|u\|^2_{L^2(\Omega)}=\rho$.

\smallskip
Over the past three decades, many researchers have considered the following Schr\"{o}dinger equation
\begin{equation}\label{25052701}
-\Delta u+\omega u=f(u)\;\; \text { in } \Omega
\end{equation}
with the mass constraint
\begin{equation}\label{25052702}
\|u\|^2_{L^2(\Omega)}=\rho.
\end{equation}
A couple  $(\omega, u)$ satisfying \eqref{25052701}-\eqref{25052702}  is usually called a {\it normalized solution} to \eqref{25052701}. Problem \eqref{25052701}-\eqref{25052702} comes from the study of standing waves for the nonlinear Schr\"{o}dinger equation
\begin{equation}\label{2207181059}
    i\psi_t-\Delta \psi =f(\psi) \;\; \text{in}~ \Omega\times (0,\infty),
\end{equation}
where $\psi$ has the form
\begin{equation}\label{2207181100}
\psi(x,t)=e^{-i\omega t} u(x),\;\; (x,t)\in \Omega\times (0,\infty),
\end{equation}
and $u$ is a real function. For problem \eqref{2207181059}-\eqref{2207181100}, the $L^2$-norm of $u$ stands for the mass of a particle, and $\omega$ stands for the frequency. There are already a lot of papers concerning the fixed frequency problem. Recently many authors are focused on the prescribed mass problem. If $f(\psi)=e^{-i\omega t}f(u)$ and the mass is prescribed, \eqref{2207181059}-\eqref{2207181100} will be reduced to the problem \eqref{25052701}-\eqref{25052702}. Here, we consider \eqref{25052701}-\eqref{25052702} under the zero Dirichlet boundary condition, that is
%When \eqref{25052701}-\eqref{25052702} subject to zero Dirichlet boundary condition, that is
\begin{equation}\label{2506161145}
\begin{cases}
\begin{aligned}
&-\Delta u+\omega u=f(u) \quad &&\text { in } \Omega, \\
&u=0  &&\text { on } \partial\Omega, \\
&\|u\|_2^2=\rho.
\end{aligned}
\end{cases}
\end{equation}
One can derive solutions to \eqref{2506161145} by looking for critical points of the associated energy functional
\begin{equation}\label{2303021513}
\begin{aligned}
E(u)=\frac{1}{2}\|\nabla u\|_2^2 -\int_{\Omega}F(u) \ud x \quad\;\; \mbox{with}\;\; F(u) =\int_0^{u}f(t)\ud t
\end{aligned}
\end{equation}
on the constraint
\[
S_{\rho}:=\{u\in H^1_0(\Omega):\|u\|^2_{L^2(\Omega)}=\rho\}.
\]

When $\Omega=\mathbb{R}^N$, according to {\it Gagliardo-Nirenberg} inequality (see \eqref{050502} below), the functional \eqref{2303021513} is coercive on $S_\rho$ if $F(u)$ grows slower than $|u|^{2+\frac{4}{N}}$ at infinity, and $E$ is unbounded from below on $S_\rho$ if the growth of $F(u)$ at infinity is faster than $|u|^{2+\frac{4}{N}}$. Accordingly the value $2+\frac{4}{N}$ is crucial for deciding the shape of $E$, which is usually called {\it mass-critical} or {\it$L^2$-critical} exponent.

Jeanjean  \cite{Jeanjean97} did a seminal work and considered a class of Sobolev supcritical and mass-subcritical problem where $f(u)$ can be chosen as $\sum\limits_{1\leq j \leq k}a_j|u|^{\sigma_j-2}u$ with $k\ge 1, a_j>0$ and $2+\frac{4}{N}<\sigma_j<\frac{2N}{(N-2)^+}$. In this case, the energy functional possesses a mountain pass geometry on $S_\rho$. He derived a Palais-Smale (for short (PS)) sequence approaching Pohozaev manifold, by using a crucial scaling transformation
\begin{equation}\label{2506151703}
u_t=t^{\frac{N}{2}}u(tx),\;\; t>0,
\end{equation}
and constructing the augmented functional
$$\widetilde{E}(u, t):=E(u_t), \;\; (u, t)\in S_\rho\times \mathbb{R}^+.$$
As a consequence, this (PS) sequence can be proved to be bounded in $H^1(\mathbb{R}^N)$. Due to this method, a lot of articles concerning normalized solutions appear in recent years, see for example \cite{BMRV2021, JZ2024, MRV2022, Soave20_JDE, Soave20_JFA} and reference therein.

\smallskip
Unfortunately, for general domain $\Omega\subsetneqq\mathbb{R}^N$, the scaling \eqref{2506151703} may escape from the energy space,  hence Jeanjean's method could be invalid. This brings the main difficulties in obtaining a bounded (PS) sequence at a minimax level. When $\Omega$ is a bounded domain, as far as we know, Struwe's monotonicity trick was valid for getting a (PS) sequence satisfying Pohozaev identity, so that this sequence is bounded. There are some papers dealing with the existence issue of  \eqref{2506161145}. We state the current results in next subsection.

\subsection{The current results in bounded domain}

In this subsection, we assume $\Omega\subset\mathbb{R}^N$ is a bounded smooth domain. For different $\mu$, $p$ and $q$, we will introduce the existence about  positive solutions of \eqref{2506161145} for the model $f(u)=\mu |u|^{p-2}u+|u|^{q-2}u$.   Let $2^{*}:=\frac{2N}{(N-2)^+}$, which denotes the Sobolev critical  exponent if $N\ge 3$, 

\medskip
\noindent {\bf $\bullet$ Case $\mu=0$, $N\ge 1$ and $2<q<2^*$}
\medskip

For the simplest case $f(u)=|u|^{q-2}u$ and $\Omega\subset\mathbb{R}^N(N\geq1)$ is a ball, Noris, Tavares and Verzini \cite{Noris14} gave the first result about the existence and non-existence of positive  solutions for problem \eqref{2506161145}. They proved that
\begin{itemize}
\item[\rm (i)] when $2<q<2+\frac{4}{N}$,  \eqref{2506161145} has a unique positive solution for every $\rho>0$, which is a global minimizer;
\item[\rm (ii)]  when $q=2+\frac{4}{N}$, there exists $\rho^{*}>0$ such that \eqref{2506161145} has a unique positive solution that is a global minimizer for $0<\rho<\rho^{*}$ and no positive solution for $\rho\ge \rho^{*}$;
\item[\rm (iii)]  when $2+\frac{4}{N}<q<2^{*}$, there exists $\rho^{*}>0$ such that \eqref{2506161145} admits a positive solution if and only if $0<\rho\le \rho^*$. Moreover, \eqref{2506161145} has at least two positive solutions for $0<\rho<\rho^*$.
\end{itemize}
\medskip
\noindent{\bf $\bullet$ Case $\mu=0$, $N\ge 3$ and $q=2^*$}
\medskip

When  $N\ge 3$ and $f(u)=|u|^{2^*-2}u$, Noris, Tavares and Verzini in \cite[Theorem 1.11]{Noris19} proved that if $ 0<\rho\le \frac{2}{N\lambda_1(\Omega)}(\frac{1}{\mathcal{S}})^{\frac{N-2}{2}}$, \eqref{2506161145} admits a positive local minimizer for some $\omega\in (-\lambda_1(\Omega), 0)$, where $\lambda_1(\Omega)$ is the first Dirichlet eigenvalue of $-\Delta$, and $\mathcal{S}$ is the best Sobolev constant, see \eqref{050501}. Furthermore, when $\Omega$ is star-shaped, Pierotti, Verzini and Yu \cite{Pierotti25} established a bounded (PS) sequence by taking advantage of monotonicity trick. By exploiting sharp estimates of the mountain pass level, they showed the strong convergence of this sequence, thus \eqref{2506161145} admits a positive mountain pass type solution for $\rho$ sufficiently small. Recently, Chang, Liu and Yan \cite{CLY2025} use Sobolev subcritical case to approach the Sobolev critical  case. They improved the results of \cite{Pierotti25} to the general bounded domain.

\smallskip
For $N\ge 6$, Lv, Zeng and Zhou \cite{LvZZ25} constructed a positive $k$-spike solution of \eqref{2506161145} in some suitable bounded domain $\Omega$ for some $k\in\mathbb{N}^+$, where $k$ depends on the Robin function and Green function of $\Omega$. Using blow-up analysis and local Pohozaev identity, they also proved that the $k$-spike solutions are locally unique.

\medskip
\noindent{\bf $\bullet$ Case $\mu\neq 0$, $N\ge 1$ and $2<p<q\le 2^*$}
\medskip

For $\mu\neq 0$,  the  problem is much less understood. Liu and Zhao \cite{LZ25} treated  the general nonlinearity $f$ with Sobolev subcritical near infinity and mass-critical or mass-supcritical near the origin, that is $2+\frac{4}{N}\le p\le q<2^*$. They obtained the existence of a local minimizer solution if  $N\ge3$, and a mountain pass type solution if $\Omega$ is a star-shaped domain and $N\ge3$. It seems that there is no literature concerning the case $2<p<2+\frac{4}{N}$ and $p\le q<2^*$ so far.

\smallskip
Song and Zou \cite{SZarxiv25} considered the existence of positive solutions to \eqref{2506161145} with $f(u)=\mu |u|^{p-2}u+|u|^{2^*-2}u~(N\geq3)$ on a bounded star-shaped domain. When $\mu=0$, or $\mu>0$ and $2+\frac{4}{N}<p<2^*$, or $\mu<0$ and $2<p<2^*$, they obtained the existence of a local minimizer for small $\rho$, which is a positive solution to \eqref{2506161145}. Using the monotonicity trick and the Pohozaev identity, when $\mu=0$, or $\mu>0$, $2+\frac{4}{N}<p<2^*$ and $N\in\{3,4,5\}$, or $\mu<0$, $2<p<2^*-1$ and $N\in\{3,4,5\}$, they established the existence of positive mountain pass type solution of \eqref{2506161145} for small $\rho$. To our knowledge,  there are no papers dealing with the case $2<p \le 2+\frac{4}{N}$ and $\mu > 0$.

\medskip
\noindent{\bf $\bullet$ Some other solutions and nonlinearities}
\medskip

Other significant results concerning normalized solutions of problem \eqref{2506161145} have been investigated and extended within different contexts, such as equations with potentials \cite{BQZ24, LM24, QZ24}, solutions with higher Morse index \cite{PD17}, Kirchhoff equations \cite{WC24}, general (non-autonomous) nonlinearities \cite{Song23}, nodal solutions \cite{DDGS2025} and sign-changing solutions \cite{SZ25}. Some results related to metric graph problems, see \cite{CLS24, Boni25} and reference therein.

\subsection{Our main results}
In this paper, we are interested in the existence and nonexistence of solutions to \eqref{041003}. Provided $q\le 2^*$, solutions for problem \eqref{041003} will be obtained as positive critical points of the energy functional
\begin{equation*}
\begin{aligned}
I(u)=\frac{1}{2}\int_{\Omega}|\nabla u|^2 \ud x
-\frac{\mu}{p}\int_{\Omega}|u|^{p} \ud x
-\frac{1}{q}\int_{\Omega}|u|^{q} \ud x,\quad \forall\, u\in H^1_0(\Omega)
\end{aligned}
\end{equation*}
on the constraint $S_{\rho}$. Motivated by the ideas in \cite{Noris19}, we first consider the existence of local minimizer of $I|_{S_\rho}$ with $q\le 2^*$. Let $\lambda_1(\Omega)$ be the first Dirichlet eigenvalue of $-\Delta$ in $\Omega$, and $\varphi_1$ be the corresponding first eigenfunction  satisfying $\|\varphi_1\|_{L^2(\Omega)}=1$ and $\varphi_1>0$ in $\Omega$. For $\alpha\geq\lambda_1(\Omega)$, we denote
\begin{equation*}
\begin{aligned}
A_{\alpha}:=&\Big\{u \in S_{\rho}:\int_{\Omega} |\nabla u|^2 \ud x\leq \rho \alpha \Big\},\\
\partial A_{\alpha}:=&\Big\{u \in S_{\rho}:\int_{\Omega} |\nabla u|^2 \ud x= \rho \alpha\Big\}.
\end{aligned}
\end{equation*}
In particular, one has $A_\alpha=\partial A_\alpha=\{\rho^{\frac12}\varphi_1, -\rho^{\frac12}\varphi_1\}$ in case $\alpha=\lambda_1(\Omega)$. We shall now consider the following local minimization problems,
\begin{equation}\label{2505261720}
\begin{aligned}
m_{\alpha}:=\inf_{A_{\alpha}} I(u),\quad \widetilde{m}_{\alpha}:=\inf_{\partial A_{\alpha}} I(u).
\end{aligned}
\end{equation}
Obviously, $m_{\lambda_1}=\widetilde{m}_{\lambda_1}$, which is attained by $\{\rho^{\frac12}\varphi_1, -\rho^{\frac12}\varphi_1\}$. For the case $p\le q<2^*$, the embeddings $H_0^1(\Omega)\subset L^p(\Omega)$ and $H_0^1(\Omega)\subset L^q(\Omega)$ are compact, hence it is obvious that $m_\alpha$ can be achieved. However, when $N\ge 3$ and $p<q=2^*$, the compactness for the minimizing sequence of $m_\alpha$ depends on the best Sobolev  constant $\mathcal{S}$, see \eqref{050501}. For any $\alpha>\lambda_{1}(\Omega)$, to prove the reachability of $m_\alpha$, we will verify in Proposition \ref{P050601} that any minimizing sequence of $m_\alpha$ is precompact in $H_0^1(\Omega)$ provided there holds
\begin{equation}\label{2505271451}
\rho<
\begin{aligned}
\frac{1}{\alpha-\lambda_1(\Omega)}\Big(\frac{2^{*}}{2}\mathcal{S}^{\frac{2^{*}}{2}}\Big)^{\frac{2}{2^{*}-2}}.
\end{aligned}
\end{equation}
Since the minimizer in $A_\alpha$ may be located on the boundary $\partial A_\alpha$, we shall rule out the occurrence of this situation. To this aim, for $\alpha>\lambda_1(\Omega)$, we will show in Section \ref{S_2} that $\widetilde{m}_{\lambda_1(\Omega)}<\widetilde{m}_{\alpha}$ if $\rho$ is small enough, which means that the minimizer is an interior point of $A_\alpha$. Our first result can be stated as follows.
\begin{theorem}\label{T041801}
Let $\mu\in \mathbb{R}$ and $2<p\le q\le  2^*$. Then there exists a $\rho_0=\rho_0(N,p,q,\Omega,\mu)>0$ such that
 for any $\rho\in(0,\rho_0)$,  \eqref{041003} has a solution $(\omega, u)\in\mathbb{R}\times S_\rho$, which is a local minimizer of $I|_{S_\rho}$.
\end{theorem}

\begin{remark}
\begin{itemize}
\item[\rm (i)] Actually, for the special cases $2<p=q\le 2^*$ and $\mu\le-1$, or $1<p\le q<2+\frac{4}{N}$, the energy functional is coercive on $S_\rho$. Therefore there exists a  global minimizer for all $\rho>0$.
\item[\rm (ii)] Theorem \ref{T041801} can improve some previous results. As we introduced before,  \cite{SZarxiv25} obtained a local minimizer but  the domain is required to be star-shaped. This constraint can be removed in our Theorem. In \cite{LZ25}, the authors only dealt with the mass-supercritical case and non-negative nonlinearities. Theorem \ref{T041801} is  valid for mass-subcritical case and sign-changing nonlinearities.
\end{itemize}
\end{remark}

In virtue of Theorem \ref{T041801}, we obtain a local minimizer of $I|_{S_\rho}$. A natural problem is: could we get another positive solution with mountain pass type?  In section \ref{2506231135}, when mass $\rho$ is sufficiently small, we can construct a mountain pass structure for the functional $I$ on $S_\rho$, thus the existence of a (PS) sequence is obvious. As we mentioned before,  the main difficulty is that we could not prove the boundedness of this sequence in $H_0^1(\Omega)$. Inspired by \cite{Pierotti25}, we will use Struwe's monotonicity trick to deal with our problem with $\mu\ge 0$. At first, we consider the functional
\[
I_{\eta}(u)=\frac{1}{2}\int_{\Omega}|\nabla u|^2 \ud x
-\eta\int_{\Omega}\left(\frac{\mu}{p} |u|^p+\frac1{q}|u|^{q}\right)  \ud x,\quad u\in S_\rho\;\;\mbox{and}\;\; \eta>0.
\]
When $\eta$ is sufficiently close to $1$, $I_\eta$  possesses the same mountain pass structure with $I$. We denote by $c_\eta$ the corresponding mountain pass level, which is non-increasing and continuous from left with respect to $\eta$. Hence $c_\eta$ is almost everywhere differentiable, and thus there exists a sequence $\eta_n\to 1$ such that $c_{\eta}$ is differentiable at $\eta_n$. Since the Sobolev subcritical case is very similar to the Sobolev critical case. Therefore we just concentrate on the proof of the case $q=2^*$.

With the help of \cite[Theorem 4.5]{Ghoussoub1993} (see Proposition \ref{22072013} below), we will show in Lemma \ref{2408312351} that there exists a bounded (PS) sequence at every $c_{\eta_n}$. By bubbling estimate, provided $\rho$ is small, this (PS) sequence is convergent. We actually get  a sequence of solutions $(\omega_{n}, u_{n})$ of
\begin{equation*}
\begin{cases}
\begin{aligned}
&-\Delta u+\omega u=\eta_n\left(\mu u^{p-1}+u^{q-1}\right),~ u>0 \quad &&\text { in } \Omega, \\
&u=0  &&\text { on } \partial\Omega, \\
&\|u\|_2^2=\rho>0.
\end{aligned}
\end{cases}
\end{equation*}
Therefore,  there holds the Pohozaev identity
\begin{equation}\label{2408291040}
\frac12\int_{\partial \Omega}|\nabla u_n|^2\sigma\cdot \nu \ud\sigma=N\left[\eta_n\int_{\Omega}\left(\frac{\mu}{p} |u_n|^p+\frac1{q}|u_n|^{q}\right)\ud x-\frac{\omega_n\rho}{2}\right]-\frac{N-2}{2}\int_{\Omega}|\nabla u_n|^2 \ud x.
\end{equation}
When $\Omega$ is star-shaped, this identity can deduce that $\{u_n\}$ is bounded in $H_0^1(\Omega)$. Still the bubbling estimate can show that $\{u_n\}$ has a  convergent subsequence.

\begin{theorem}\label{T041805}
Let $\mu\geq 0$,  $\Omega$ is a star-shaped domain. If
\begin{itemize}
\item[\rm (i)] $N\ge 1$, $2<p<2+\frac{4}{N}< q< 2^*$ or $N\geq 4$, $2<p<q=2^*$.
\item[\rm (ii)] $N=3$, $p\in (2, \frac{10}{3}]\cup (\frac{14}{3}, 6)$, $q=2^*$,  and 
\begin{equation}\label{25092601}
\lambda_1(B_{R_{\Omega}})<
\begin{cases}
\frac{4}{3}\lambda_1(\Omega)\quad\; &\mbox{if}\;\; 2<p\le \frac{10}{3},\\
\left(2-\frac{8}{3(p-2)}\right)\lambda_1(\Omega)\quad\; &\mbox{if}\;\;  \frac{14}{3}< p<6,\\
\end{cases}
\end{equation}
where $R_\Omega$ is the inradius of $\Omega$. 
\end{itemize}

Then for  sufficiently small $\rho>0$, problem \eqref{041003} has at least two solutions.
\end{theorem}

\begin{remark}
\begin{itemize}
\item[\rm (i)] Note that for the case $\mu=0$, up to a scaling, Pierotti, Verzini and Yu \cite{Pierotti25}  used the mass $\rho$ as a parameter in monotonicity trick. However, for the combined  nonlinearities, we need to add an extra parameter $\eta$ to take use of  monotonicity trick. Then the mass $\rho$ and $\eta$ shall be controlled simultaneously in constructing mountain pass structure. On the other hand, the estimates for the energy of solutions to \eqref{041003} is also very different, see for example, Lemma \ref{2509121227}. 

\item[\rm (ii)] Compared with \cite{SZarxiv25}, Theorem \ref{T041805} can obtain two positive solutions for all  $N\ge 3$, however the same result in \cite{SZarxiv25} holds  only for $N\in \{3, 4, 5\}$. Hence we improve the results of \cite{SZarxiv25}.

\item[\rm (iii)]  When $N=3$, we just derive the existence for $p\in (2, \frac{10}{3}]\cup (\frac{14}{3}, 6)$, this may be a technical restriction. The condition \eqref{25092601} can be satisfied if $\Omega$ is a small perturbation of balls in for example $C^1$ topology.
\end{itemize}
\end{remark}

In the following, we will consider the problem \eqref{041003} with  $q\ge 2^*$ in a convex domain. Assume $u$ is a solution of \eqref{041003}. Recall the moving-plane method in \cite{GNN1979}, if all sectional curvatures at every point of $\partial\Omega$ are positive, we can use moving-plane method near the boundary. As de Figueiredo,  Lions and Nussbaum \cite{dLN1982} said, the moving-plane method is still valid for general convex domains. We will see in Section \ref{S_4} that $u$ is monotonic along some directions around the boundary $\partial\Omega$. For the sake of the nonexistence of solutions to \eqref{041003}, we first estimate the range of the frequency $\omega$ in Lemma \ref{2410040137}. Thanks to the monotonicity of $u$ near the boundary, when $\rho$ is large enough, the mass will be concentrated on the center of $\Omega$. This can deduce that $\omega$ will converge to infinity as $\rho\to\infty$, which will reach a contradiction.
\begin{theorem}\label{T051401}
Let $1<p\le 2$, $q\ge \max\{2^*, 3\}$, $\mu> 0$ and $\Omega$ is a bounded convex domain with smooth boundary. Then problem \eqref{041003} has no positive normalized solution provided that $\rho>0$ is sufficiently large.
\end{theorem}

\begin{remark}
When $p=2$ and $q=2^*$, this problem is the classical Br\'ezis-Nirenberg type problem. Many results (see for instance \cite{Pierotti25, SZarxiv25, CLY2025}) only obtain the existence for small mass $\rho$. It seems that there is no literature concerning nonexistence of positive normalized  solutions for Sobolev critical and supercritical problem. Notice that $2^*\ge 3$ means $3\le N\le 6$, thus Theorem \ref{T051401} can deduce the nonexistence result of Br\'ezis-Nirenberg type problem for $3\le N\le 6$. 
\end{remark}

Theorem \ref{T051401} deduce that the nonexistence result for Br\'ezis-Nirenberg type problem only holds in $3\le N\le 6$. We would like to remark that this may be a technical restriction. When the parameters $\mu= 0$ and $q=2^*$, for a general bounded convex domain, it is an open problem in this paper whether there is no any positive solution of \eqref{041003} for  $N\ge 7$ and sufficiently large mass. This problem may be expected to be true because we will prove a dichotomy result for this problem in a ball $\Omega=B_1$, that is
\begin{theorem}\label{2505261703}
Assume that $N\ge 3$, $\mu=0$, $q=2^*$ and $\Omega$ is a ball, then there is some $\rho^*>0$ such that
\begin{itemize}
\item[\rm (i)] \eqref{041003} has at least two solutions if $0<\rho<\rho^*$;
\item[\rm (ii)] \eqref{041003} has at least one solution if $\rho=\rho^*$;
\item[\rm (iii)] \eqref{041003} has no solutions if $\rho>\rho^*$.
\end{itemize}
\end{theorem}

\smallskip
In fact, when $\mu=0$ and $\Omega=B_1(0)$, the problem \eqref{041003} can be reduced to consider the tendency of $L^2$-norm of the unique solution to the classical Br\'{e}zis-Nirenberg problem
\begin{equation}\label{060909}
\begin{cases}
-\Delta u-\lambda u=u^{2^{*}-1},~ u>0 \quad\text { in } B_1(0), \\
u\in H^1_0(B_1(0)).
\end{cases}
\end{equation}
A well-known result in \cite{Brezis83} said
\begin{itemize}
\item[(i)] when $N>3$, \eqref{060909} has a solution if and only if $\lambda\in\big(0,\lambda_1(B_1(0))\big)$;
\item[(ii)] when $N=3$, \eqref{060909} has a solution if and only if $\lambda\in\left(\frac{\lambda_1(B_1(0))}{4},\lambda_1(B_1(0))\right)$.
\end{itemize}
Another important argument for us is that the solution of \eqref{060909} is unique and radially symmetric, which was proved in \cite{Zhang92}. Thus the solution of \eqref{060909} is equivalent to a solution of an ordinary differential equation. By the continuity theorem of solution on initial condition and parameter, the $L^2$-norm of solution $u_\lambda$ to \eqref{060909} is continuous with respect to $\lambda$. In virtue of \cite{Zhang92}, we can show in Lemma \ref{060921} that when $N>3$, $\|u_\lambda\|_2\to 0$ as $\lambda\to 0$ and $\lambda\to \lambda_1(B_1(0))$, and when $N=3$, $\|u_\lambda\|_2\to 0$ as $\lambda\to {\frac{\lambda_1(B_1(0))}{4}}$ and $\lambda\to \lambda_1(B_1(0))$. These arguments can prove Theorem \ref{2505261703}.

\begin{remark}
Our results above provide complete answers for the normalized solutions of Br\'{e}zis-Nirenberg problem in a unit ball. However, when $\Omega$ is not a ball, our method will be invalid since there might be multiple solutions for fixed $\lambda$ if the topology of $\Omega$ is not simple. 
\end{remark}

\medskip
Our paper is organized as follows. In Section \ref{S_2}, we are devoted to proving the existence of local minimizer. Section \ref{2506231135} is dedicated to establishing the mountain pass structure and using monotonicity trick to show the existence of the second positive solution. In Section \ref{S_4}, we consider the nonexistence issue for the case that $1<p\le 2$, $q\ge \max\{2^*, 3\}$ and $\Omega$ is convex domain. In Section \ref{S_5}, we will give a dichotomy result for normalized solutions for Br\'{e}zis-Nirenberg problem in a ball.

\section{Existence~of~local~minimizer}\label{S_2}
In this section, we are mainly focused on the existence of minimizer for $m_\alpha$ given in \eqref{2505261720}. At the beginning, we recall some useful inequalities. First, for every $N \geq 3$, there exists an optimal constant $\mathcal{S}$ depending only on $N$, such
that
\begin{equation}\label{050501}
\begin{aligned}
\mathcal{S}\|u\|^{2}_{L^{2^{*}}(\mathbb{R}^N)} \leq\|\nabla u\|^{2}_{L^{2}(\mathbb{R}^N)}\quad \forall\, u \in D^{1,2}(\mathbb{R}^{N}), \quad \text{ (Sobolev inequality) }
\end{aligned}
\end{equation}
where $D^{1,2}(\mathbb{R}^{N})$ denotes the completion of $C^{\infty}_{c}(\mathbb{R}^{N})$ with respect to the norm $\|u\|_{D^{1,2}(\mathbb{R}^{N})}:=\|\nabla u\|_{L^{2}(\mathbb{R}^{N})}$. For $N \geq 1$, $2\leq q< 2^{*}$ and $u \in H^{1}(\mathbb{R}^N)$,
\begin{equation}\label{050502}
\begin{aligned}
\|u\|_{L^q(\mathbb{R}^N)}\leq \mathcal{C}_{N,q}\|\nabla u\|^{\gamma_{q}}_{L^{2}(\mathbb{R}^{N})}
\|u\|^{1-\gamma_{q}}_{L^{2}(\mathbb{R}^{N})}, \quad \text{ (Gagliardo-Nirenberg inequality) }
\end{aligned}
\end{equation}
where $\gamma_{q}:=\frac{N(q-2)}{2q}$. We emphasize that \eqref{050501} and \eqref{050502} also hold
in $H^1_0(\Omega)$, for any bounded domain $\Omega$, with the same constant $\mathcal{C}_{N,q}$.

\begin{proposition}\label{P050601}
Let $N\geq 3$, $2<p<q=2^*$, $\mu\in \mathbb{R}$ and $\alpha>\lambda_{1}(\Omega)$. If \eqref{2505271451} holds, then any minimizing sequence associated to $m_{\alpha}$ {\rm(}given in \eqref{2505261720}{\rm)} is relatively compact in $A_{\alpha}$. In particular, $m_{\alpha}$ can be achieved.
\begin{proof}
Let $\{u_n\}\subset H^1_0(\Omega)$ be a minimizing sequence for $m_{\alpha}$. That is,
\begin{equation}\label{040503}
\begin{cases}
\|u_n\|^{2}_{L^{2}(\Omega)}=\rho, \\
\|\nabla u_n\|^{2}_{L^{2}(\Omega)}\leq\rho\alpha, \\
m_{\alpha}\leq I(u_n)\leq m_{\alpha}+o(1)
\end{cases}
\end{equation}
as $n \rightarrow \infty$. Up to a subsequence, there exists $u_0\in H^1_0(\Omega)$ such that
\begin{equation*}
\begin{cases}
u_n\rightharpoonup u_0 \text{ in } H^1_0(\Omega), \\
\|\nabla u_0\|^{2}_{L^{2}(\Omega)}\leq\liminf\limits_{n\rightarrow \infty}\|\nabla u_n\|^{2}_{L^{2}(\Omega)}, \\
\|u_0\|^{2}_{L^{2}(\Omega)}=\rho.
\end{cases}
\end{equation*}
So, $u_0\in A_{\alpha}$ and $I(u_0)\geq m_{\alpha}$.

Let $\tilde{u}_n=u_n-u_0$. Then, we extract a subsequence, still denoted by $\{\tilde{u}_n\}$, such that
\begin{equation}\label{2505211318}
\begin{aligned}
\tilde{u}_n&\rightharpoonup 0 ~~\text{ in } H^1_0(\Omega), \\
\tilde{u}_n&\rightarrow 0 ~~\text{ in } L^{r}(\Omega) \text{ for } r\in[1,2^{*}),\\
\tilde{u}_n&\rightarrow 0 ~~\text{ a.e.}\text{ in }\Omega
\end{aligned}
\end{equation}
as $n \rightarrow \infty$. By weak convergence and Br\'{e}zis-Lieb lemma, we deduce that
\begin{equation}\label{040504}
\begin{aligned}
\|\nabla (\tilde{u}_n+u_0)\|^2_{L^{2}(\Omega)}
&=\|\nabla \tilde{u}_n\|^2_{L^{2}(\Omega)}+\|\nabla u_0\|^2_{L^{2}(\Omega)}+o(1), \\
\|\tilde{u}_n+u_0\|^{2^{*}}_{L^{2^{*}}(\Omega)}
&=\|\tilde{u}_n\|^{2^{*}}_{L^{2^{*}}(\Omega)}+\|u_0\|^{2^{*}}_{L^{2^{*}}(\Omega)}+o(1).
\end{aligned}
\end{equation}
By \eqref{040504}, we get
\begin{equation*}
\begin{aligned}
I(u_n)&=\frac{1}{2}\int_{\Omega}|\nabla u_n|^2 \ud x
-\frac{\mu}{p}\int_{\Omega}|u_n|^{p} \ud x
-\frac{1}{2^{*}}\int_{\Omega}|u_n|^{2^{*}} \ud x\\
&=I(\tilde{u}_n)+I(u_0)+o(1).
\end{aligned}
\end{equation*}
Using \eqref{040503} and $I(u_0)\geq m_{\alpha}$, we have $I(\tilde{u}_n)\leq o(1)$ as $n \rightarrow \infty$. From \eqref{050501} and \eqref{2505211318}, it follows that
\begin{equation}\label{050604}
\begin{aligned}
\int_{\Omega}|\nabla \tilde{u}_n|^2 \ud x
&\leq\frac{2\mu}{p}\int_{\Omega}|\tilde{u}_n|^{p} \ud x
+\frac{2}{2^{*}}\int_{\Omega}|\tilde{u}_n|^{2^{*}} \ud x+o(1)\\
&\le\frac{2}{2^{*}}\mathcal{S}^{-\frac{2^{*}}{2}}\|\nabla \tilde{u}_n\|^{2^{*}}_{L^{2}(\Omega)}+o(1).
\end{aligned}
\end{equation}
Since $\{\|\nabla \tilde{u}_n\|^{2}_{L^{2}(\Omega)}\}$ is a bounded sequence, up to a subsequence, we distinguish the two cases
\begin{equation*}
\begin{aligned}
\text{either }\mathrm{(a)}~~\|\nabla \tilde{u}_n\|^{2}_{L^{2}(\Omega)}\rightarrow 0 \quad
\text{ or } \mathrm{(b)}~~\|\nabla \tilde{u}_n\|^{2}_{L^{2}(\Omega)}\rightarrow \ell>0.
\end{aligned}
\end{equation*}
If $\mathrm{(b)}$ holds,
%\begin{equation*}
%\begin{aligned}
%\|\nabla \tilde{u}_n\|^{2}_{L^{2}(\Omega)}
%\geq\Big(\frac{2^{*}}{2}\mathcal{S}^{\frac{2^{*}}{2}}\Big)^{\frac{2}{2^{*}-2}}+o(1).
%\end{aligned}
%\end{equation*}
by \eqref{040503}-\eqref{050604} and $\|\nabla u_0\|^{2}_{L^{2}(\Omega)}\geq\lambda_{1}(\Omega)\rho$ with $u_0\in S_{\rho}$, we get that
\begin{equation*}
\begin{aligned}
\Big(\frac{2^{*}}{2}\mathcal{S}^{\frac{2^{*}}{2}}\Big)^{\frac{2}{2^{*}-2}}
\leq\|\nabla \tilde{u}_n\|^{2}_{L^{2}(\Omega)}+o(1)
&=\|\nabla u_n\|^{2}_{L^{2}(\Omega)}
-\|\nabla u_0\|^{2}_{L^{2}(\Omega)}+o(1)\\
&\leq(\alpha-\lambda_{1}(\Omega))\rho+o(1),
\end{aligned}
\end{equation*}
which contradicts \eqref{2505271451}. Thus, $\|\nabla \tilde{u}_n\|^{2}_{L^{2}(\Omega)}\rightarrow 0$. That is, $u_n\rightarrow u_0$ in $H^1_0(\Omega)$ and $I(u_0)=m_{\alpha}$.
\end{proof}
\end{proposition}
\begin{lemma}\label{L050602}
Let $2<p<q\le 2^*$, $\mu\in \mathbb{R}$. Assume that there exist $\alpha_2>\alpha_1\ge\lambda_1(\Omega)$ such that 
\begin{equation}\label{050603}
\begin{aligned}
 \widetilde{m}_{\alpha_1}<\widetilde{m}_{\alpha_2}.
\end{aligned}
\end{equation}
If $N\ge 3$ and $q=2^*$, we need the extra assumption that \eqref{2505271451} holds with $\alpha= \alpha_2$. Then $m_{\alpha_2}<\widetilde{m}_{\alpha_2}$, and $m_{\alpha_2}$ can be achieved by a positive solution of problem \eqref{041003}.
\begin{proof}
By \eqref{050603}, we have
\begin{equation}\label{2510122022}
\begin{aligned}
m_{\alpha_2} = \inf\{\widetilde{m}_{\alpha}:\lambda_1(\Omega)\leq \alpha \leq \alpha_2\}
\leq \widetilde{m}_{\alpha_1}
<\widetilde{m}_{\alpha_2}.
\end{aligned}
\end{equation}
When $q<2^*$, since  the embeddings $H_0^1(\Omega)\subset L^p(\Omega)$ and $H_0^1(\Omega)\subset L^q(\Omega)$ are compact, it is clear that there exists some $u_0\in A_{\alpha_2}$ such that $I(u_0)=m_{\alpha_2}$. When $q=2^*$, from Proposition \ref{P050601}, we know that $m_{\alpha_2}$ can  be also achieved by some $u_0\in A_{\alpha_2}$. Next, it follows from \eqref{2510122022} that $u_0\in A_{\alpha_2}\setminus \partial A_{\alpha_2}$.
\end{proof}
\end{lemma}
\begin{proof}[\bf Proof of Theorem \ref{T041801}]
By $\sqrt{\rho}\varphi_1 \in \partial A_{\lambda_1(\Omega)}$,  then we have
%$\widetilde{m}_{\lambda_1(\Omega)}\leq \frac{\rho}{2}\lambda_1(\Omega)$.
\begin{equation}\label{2505211445}
\begin{aligned}
\widetilde{m}_{\lambda_1(\Omega)}\leq I(\sqrt{\rho}\varphi_1)
=\frac{\lambda_1(\Omega)}{2}\rho
-\frac{\mu}{p}\|\varphi_1\|^p_{L^{p}(\Omega)}\rho^{\frac{p}{2}}
-\frac{1}{2^{*}}\|\varphi_1\|^{q}_{L^{q}(\Omega)}\rho^{\frac{q}{2}}.
\end{aligned}
\end{equation}
This proof will be split into three cases.
\smallskip

{\it Case $N\ge 3$, $2<p<q=2^*$}. For $u\in \partial A_{\alpha}$, by Sobolev inequality and Gagliardo-Nirenberg inequality, it holds that
\begin{equation}\label{25052719099}
\begin{aligned}
I(u)&=\frac{1}{2}\int_{\Omega}|\nabla u|^2 \ud x
-\frac{\mu}{p}\int_{\Omega}|u|^{p} \ud x
-\frac{1}{2^{*}}\int_{\Omega}|u|^{2^{*}} \ud x\\
&\geq \frac{1}{2}\int_{\Omega}|\nabla u|^2 \ud x
-\frac{|\mu|}{p}\mathcal{C}^{p}_{N,p}\|\nabla u\|^{p\gamma_{p}}_{L^{2}(\Omega)}
\|u\|^{p(1-\gamma_{p})}_{L^{2}(\Omega)}
-\frac{1}{2^{*}}\mathcal{S}^{-\frac{2^{*}}{2}}\|\nabla u\|^{2^{*}}_{L^{2}(\Omega)}\\
&=\frac{\alpha}{2}\rho
-\frac{|\mu|}{p}\mathcal{C}^{p}_{N,p}\alpha^{\frac{p\gamma_{p}}{2}}\rho^{\frac{p}{2}}
-\frac{1}{2^{*}}\mathcal{S}^{-\frac{2^{*}}{2}}\alpha^{\frac{2^{*}}{2}}\rho^{\frac{2^{*}}{2}}.
\end{aligned}
\end{equation}
Thus, we get that
\begin{equation}\label{2505211446}
\begin{aligned}
\widetilde{m}_{\alpha}\geq
\frac{\alpha}{2}\rho
-\frac{|\mu|}{p}\mathcal{C}^{p}_{N,p}\alpha^{\frac{p\gamma_{p}}{2}}\rho^{\frac{p}{2}}
-\frac{1}{2^{*}}\mathcal{S}^{-\frac{2^{*}}{2}}\alpha^{\frac{2^{*}}{2}}\rho^{\frac{2^{*}}{2}}.
\end{aligned}
\end{equation}

For each $\alpha>\lambda_1(\Omega)>0$, we consider the function
\begin{equation*}
\begin{aligned}
\hat{f}_{\alpha}(\rho)
=&\frac{\alpha-\lambda_1(\Omega)}{2}\rho
+{\frac{1}{p}\left(\mu\|\varphi_1\|^p_{L^{p}(\Omega)}-|\mu|\mathcal{C}^{p}_{N,p}\alpha^{\frac{p\gamma_{p}}{2}}\right)\rho^{\frac{p}{2}}}
+\frac{1}{2^{*}}\big(\|\varphi_1\|^{2^{*}}_{L^{2^{*}}(\Omega)}
-\mathcal{S}^{-\frac{2^{*}}{2}}\alpha^{\frac{2^{*}}{2}}\big)\rho^{\frac{2^{*}}{2}}.\\
\end{aligned}
\end{equation*}
By $\alpha>\lambda_1(\Omega)$ and $2<p<2^*$, we deduce that there exists $\rho_0=\rho_0(N,p,\Omega, \mu, \alpha)>0$ such that  for any $\rho\in(0,\rho_0)$, $\hat{f}_{\alpha}(\rho)>0$.
Thus, it follows from \eqref{2505211445}, \eqref{2505211446} that
%Thus, for $\rho>0$ sufficiently small, we have
\begin{equation*}
\begin{aligned}
\widetilde{m}_{\alpha}
&\geq
\frac{\alpha}{2}\rho
-\frac{|\mu|}{p}\mathcal{C}^{p}_{N,p}\alpha^{\frac{p\gamma_{p}}{2}}\rho^{\frac{p}{2}}
-\frac{1}{2^{*}}\mathcal{S}^{-\frac{2^{*}}{2}}\alpha^{\frac{2^{*}}{2}}\rho^{\frac{2^{*}}{2}}\\
&>\frac{\lambda_1(\Omega)}{2}\rho
-\frac{\mu}{p}\|\varphi_1\|^p_{L^{p}(\Omega)}\rho^{\frac{p}{2}}
-\frac{1}{2^{*}}\|\varphi_1\|^{2^{*}}_{L^{2^{*}}(\Omega)}\rho^{\frac{2^{*}}{2}}
\geq\widetilde{m}_{\lambda_1(\Omega)},\quad \forall\,\rho\in(0,\rho_0).
\end{aligned}
\end{equation*}
Using Lemma \ref{L050602}, we get that $m_{\alpha}$ is achieved by a positive solution of problem \eqref{041003}.

\smallskip
{\it Case $N\ge 1$, $2<p<q<2^*$}. In this case,  the Sobolev inequality can be replaced by Gagliardo-Nirenberg inequality, all steps above can proceed.

\smallskip
{\it Case $N\ge 1$, $2<p=q\le 2^*$}. If $\mu>-1$,  this situation is identical to the above cases with $\mu=0$. If $\mu\le -1$, the associated energy is coercive, and it is simple to obtain a global minimizer. Moreover, the global minimizer exists for all $\rho>0$.
\end{proof}

\section{Existence of second positive solution with mountain pass type}\label{2506231135}
In this section, we will assume $\mu\ge 0$. Our aim is to find the second positive solution to \eqref{041003}. Theorem \ref{T041801} has given a local minimizer. The second positive solution will be found by constructing mountain pass structure. For simplicity, we will concentrate on the proof of the case $N\ge 3$ and $q=2^*$. Hence without confusions, we will replace $q$ by $2^*$.
\subsection{Mountain pass structure and monotonicity method}\label{section3.1}
In this subsection, we will construct a mountain pass structure and use the monotonicity trick as in \cite{Jeanjean99} to derive a bounded (PS) sequence. We set $f(u)=\mu |u|^{p-2}u+|u|^{2^*-2}u$ with $\mu\ge 0$ and the functional
\[
I_{\eta}(u)=\frac{1}{2}\int_{\Omega}|\nabla u|^2 \ud x
-\eta\int_{\Omega}F(u) \ud x,\quad \forall\, u\in S_\rho\;\;\mbox{and}\;\; \eta>0,
\]
where $F(u)=\frac{\mu}{p}|u|^p+\frac{1}{2^*}|u|^{2^*}$.  

The first step is to find two suitable end points. Without loss of generality, we can suppose $0\in\Omega$.  For the sake of bubbling estimate, we consider the functions
\[
U_\varepsilon(x)=\frac{\varphi [N(N-2)\varepsilon^2]^{\frac{N-2}{4}}}{[\varepsilon^2+|x|^2]^{\frac{N-2}{2}}},
\]
where $\varphi\in C_c^{\infty}(\Omega)$ is a cut-off function such that $0\le \varphi\le 1$ and $\varphi=1$ in some neighborhood of $0$. The following estimates of $U_\varepsilon$ can be found in \cite{Brezis83} and \cite[Lemma 7.1]{JL2022}. As $\varepsilon\to 0$, there hold that
\begin{equation}\label{2509112114}
\begin{aligned}
\|\nabla U_\varepsilon\|_2^2&={\mathcal{S}}^{\frac{N}{2}}+O(\varepsilon^{N-2}),\\
\|U_\varepsilon\|_{2^*}^{2^*}&={\mathcal{S}}^{\frac{N}{2}}+O(\varepsilon^{N}),\\
\end{aligned}
\end{equation}
and
\begin{equation}\label{2509112115}
\begin{aligned}
\|U_\varepsilon\|_{2}^{2}&=
\begin{cases}
\begin{aligned}
&c\varepsilon^2+O(\varepsilon^{N-2})\quad &&\mbox{if}\; N\ge 5,\\
&c\varepsilon^2\ln{\varepsilon}+O(\varepsilon^2)\quad &&\mbox{if}\; N=4,\\
&c\varepsilon+O(\varepsilon^2)\quad &&\mbox{if}\; N=3,\\
\end{aligned}
\end{cases}
\end{aligned}
\end{equation}
where ${\mathcal{S}}$ is the best Sobolev constant and $c$ is a positive constant. We define
\begin{equation}\label{2509121102}
v_\varepsilon=\frac{\sqrt{\rho}U_\varepsilon}{\|U_\varepsilon\|_2},
\end{equation}
which is located on $S_\rho$. We set the parameter $\alpha>\frac{\|\nabla U_1\|_2^2}{\|U_1\|_2^2}$, then $\|\nabla v_1\|_2^2<\alpha \rho$, so $v_1\in A_\alpha\backslash \partial A_\alpha$.  

In the other hand, similar to \eqref{25052719099}, we have
\begin{equation}\label{2509252053}
I_\eta(u)\ge \rho\left[\frac{\alpha}{2}
-\frac{\eta\mu}{p}\mathcal{C}^{p}_{N,p}\alpha^{\frac{p\gamma_{p}}{2}}\rho^{\frac{p}{2}-1}
-\frac{\eta}{2^{*}}\mathcal{S}^{-\frac{2^{*}}{2}}\alpha^{\frac{2^{*}}{2}}\rho^{\frac{2^{*}}{2}-1}\right],\quad \forall\, u\in \partial A_\alpha.
\end{equation}
When  $\rho$ is small enough, and $\eta$ approaches 1, we can get
\begin{equation}\label{2509252052}
\underset{u\in \partial A_\alpha}{\inf}I_\eta(u)>I_\eta(v_1).
\end{equation}
Combining \eqref{2509112114}-\eqref{2509121102}, one can verify
\begin{equation}\label{2509252054}
\|\nabla v_\varepsilon\|_2^2=\rho\left({\mathcal{S}}^{\frac{N}{2}}+O(\varepsilon^{N-2})\right)\|U_\varepsilon\|^{-2}_2,
\end{equation}
and
\begin{equation}\label{2509252055}
I_\eta(v_\varepsilon)\le  \rho\left[\frac{1}{2}
\left({\mathcal{S}}^{\frac{N}{2}}+O(\varepsilon^{N-2})\right)\|U_\varepsilon\|^{-2}_2
-\frac{\eta\rho^{\frac{2^*}{2}-1}}{2^*}\left({\mathcal{S}}^{\frac{N}{2}}+O(\varepsilon^{N})\right)\|U_\varepsilon\|^{-2^*}_2\right]
\end{equation}
as $\varepsilon\to 0$. By \eqref{2509112114}-\eqref{2509252055}, we can obtain the following lemma.
\begin{lemma}\label{2509112067}
 For $N\ge 3$, $2<p<2^*$, $\alpha>\frac{\|\nabla U_1\|_2^2}{\|U_1\|_2^2}$, there exist  $\widetilde{\rho}_0<\rho_0$ and $\eta_0\in (0,1)$ such that for any $\rho\in (0, \widetilde{\rho}_0)$, we can find  
  a $\varepsilon_0\in (0, 1)$ such that for any  $\eta\in (\eta_0, 1]$ and $\varepsilon\in (0, \varepsilon_0]$, we have
 \[
 v_1\in A_\alpha,\quad v_{\varepsilon}\in S_\rho\backslash A_\alpha\quad\mbox{and}\quad \max\{I_\eta(v_1), I_\eta(v_{\varepsilon})\}<\underset{u\in \partial A_\alpha}{\inf}I_\eta(u).
 \]
where where $\rho_0$ and $v_\varepsilon$  are given in Theorem \ref{T041801} and   \eqref{2509121102}.
\end{lemma}

Apparently,  we can construct a mountain pass structure for the functional $I_\eta$ constrained on $S_\rho$. Using the notations of Lemma \ref{2509112067}, for any $\rho< \widetilde{\rho}_0$, we define
\[
\Gamma=\{\gamma \in C([0,1], S_\rho): \gamma (0)=v_1, \gamma (1)=v_{\varepsilon_0}\}
\]
 and for $\eta\in (\eta_0, 1]$,
\begin{equation}\label{2509112066}
c_\eta= \underset{\gamma\in\Gamma}{\inf}\underset{t\in [0,1]}{\max} {I_\eta(\gamma(t))}.
\end{equation}
Notice that $\varepsilon_0$ depends on the choice of $\rho$, but  the mountain pass level $c_\eta$ is invariant if the end point $v_{\varepsilon_0}$ is replaced by $v_{\varepsilon}$ with $\varepsilon<\varepsilon_0$.

We define
\[
\gamma_{0}(t)=v_{(1-t)+t{\varepsilon_0}},
\]
thus $\gamma_0\in\Gamma$. For $\rho<\widetilde{\rho}_0$,  it follows from Lemma \ref{2509112067} that 
\begin{equation}\label{2509112056}
\underset{t\in [0, 1]}{\max}I_\eta(\gamma_0(t))=\underset{\varepsilon\in (\varepsilon_0, 1]}{\max}I_\eta(v_\varepsilon)=\underset{\varepsilon\in (0, 1]}{\max}I_\eta(v_\varepsilon).
\end{equation}
 Now let us give the estimate for the upper bound of $c_\eta$.
\begin{lemma}\label{2509121247}
Assume that $N\ge 3$ and in addition $\overline{B}_R\subset\Omega$ for some $R>0$ if $N=3$. Then there exist $\widehat{\rho}_0>0$ and $C>0$ such that when $\rho\in (0,\widehat{\rho}_0)$, we have
\begin{equation} \label{2509121246}
c_\eta\le \frac{1}{N}{\mathcal{S}}^{\frac{N}{2}}\eta^{1-\frac{N}{2}}+h_\eta(\rho)
\end{equation}
where
\[
h_\eta(\rho)=
\begin{cases}
C\rho^{\frac{N-2}{2}}\eta^{\frac{(N-2)(N-4)}{4}}\quad\; &\mbox{if}\; N\ge 5,\\
C\rho|\ln{(\rho\eta)}|^{-1}\quad\; &\mbox{if}\; N= 4,\\
\frac{1}{4}\lambda_1(B_R)\rho+C\rho^{2}\eta^{\frac12}\quad\; &\mbox{if}\; N=3.\\
\end{cases}
\]
\end{lemma}
\begin{proof}
By \eqref{2509112066} and \eqref{2509112056}, we have
\[
c_\eta\le \underset{\varepsilon\in (0, 1]}{\sup}I_\eta(v_\varepsilon).
\]
By the definition of $v_\varepsilon$, we have
\begin{equation}\label{2509120207}
\begin{aligned}
\underset{\varepsilon\in (0, 1]}{\max}I_\eta(v_\varepsilon)&=\underset{\varepsilon\in (0, 1]}{\max} \left(\frac{1}{2}\|\nabla v_\varepsilon\|_2^2-\frac{\mu\eta}{p}\|v_\varepsilon\|_p^p-\frac{\eta}{2^*}\|v_\varepsilon\|_{2^*}^{2^*}\right)\\
&\le\underset{\varepsilon\in (0, 1]}{\max} \left(\frac{1}{2}\|\nabla v_\varepsilon\|_2^2-\frac{\eta}{2^*}\|v_\varepsilon\|_{2^*}^{2^*}\right)\\
&=\underset{\varepsilon\in (0, 1]}{\max} \left(\frac{\rho}{2}\frac{\|\nabla U_\varepsilon\|_2^2}{\| U_\varepsilon\|_2^2}
-\eta\frac{\rho^{\frac{2^*}{2}}}{2^*}\frac{\| U_\varepsilon\|_{2^*}^{2^*}}{\| U_\varepsilon\|_2^{2^*}}\right)\\
&=\rho\underset{\varepsilon\in (0, 1]}{\max} \left(\frac{1}{2}\frac{\|\nabla U_\varepsilon\|_2^2}{\| U_\varepsilon\|_2^2}
-\frac{\eta\rho^{\frac{2^*}{2}-1}}{2^*}\frac{\| U_\varepsilon\|_{2^*}^{2^*}}{\| U_\varepsilon\|_2^{2^*}}\right).
\end{aligned}
\end{equation}
Using \cite[Lemma 3.7]{Pierotti25}, we can deduce that
\begin{equation}\label{2509120208}
\underset{\varepsilon\in (0, 1]}{\max} \left(\frac{1}{2}\frac{\|\nabla U_\varepsilon\|_2^2}{\| U_\varepsilon\|_2^2}
-\frac{\eta\rho^{\frac{2^*}{2}-1}}{2^*}\frac{\| U_\varepsilon\|_{2^*}^{2^*}}{\| U_\varepsilon\|_2^{2^*}}\right)\le \frac{1}{N}{\mathcal{S}}^{\frac{N}{2}}\rho^{-1}\eta^{1-\frac{N}{2}}+\widetilde{h}_\eta(\rho)
\end{equation}
where
\[
\widetilde{h}_\eta(\rho)=
\begin{cases}
C\rho^{\frac{N-4}{2}}\eta^{\frac{(N-2)(N-4)}{4}}\quad\; &\mbox{if}\; N\ge 5,\\
C|\ln{(\rho\eta)}|^{-1}\quad\; &\mbox{if}\; N= 4,\\
\frac{1}{4}\lambda_1(B_R)+C\rho\eta^{\frac12}\quad\; &\mbox{if}\; N=3.\\
\end{cases}
\]
Combining \eqref{2509120207} and \eqref{2509120208}, we can get the desired result.
\end{proof}

It is clear that $c_{\eta}$ is non-increasing, so $c_{\eta}$ is almost everywhere differentiable with respect to $\eta$.
For convenience, we denote
\[
S_\rho^+:=\{u\in S_\rho: u\ge 0\}.
\]

Next, we will use monotonicity trick to prove
\begin{proposition}\label{2408310238}
\begin{itemize}
\item[\rm (i)] For any $\eta\in (\eta_0, 1]$ where $\eta_0$ is defined in Lemma \ref{2509112067}, then $c_\eta$ is continuous from left.
\end{itemize}
\smallskip
 Moreover, assume that $c_{\eta}$ is differentiable at some $\eta\in (\eta_0, 1)$. Then  we have
\begin{itemize}
\item[\rm (ii)] for any $B>0$, there exist $C>0$ {\rm(}depending on $B$, $c_\eta$ and $c'_\eta:=\frac{\rm d}{\ud\eta}c_{\eta}${\rm)} and $\eta_1\in (\eta_0, \eta)$ such that for any $\widetilde{\eta}\in (\eta_1, \eta)$ and
\begin{equation}\label{2509180016}
\widetilde{\gamma}\in \{\gamma\in\Gamma: \max_{t\in [0, 1]}I_{\widetilde{\eta}}(\gamma(t))\le c_{\widetilde{\eta}}+B(\eta-\widetilde{\eta})\},
\end{equation}
 if $I_{\eta}(\widetilde{\gamma}(t))\ge c_\eta-(\eta-\widetilde{\eta})$ for some $t\in [0, 1]$, then $\|\widetilde{\gamma}(t)\|_{H_0^1(\Omega)}\le C$;
\item[\rm (iii)] there exists $\eta_1\in (\eta_0, \eta)$ such that for $\widetilde{\eta}\in (\eta_1, \eta)$, one can find a $\gamma\in \Gamma$ such that ${\rm Im}(\gamma)\subset S_\rho^+$ and $\underset{t\in [0, 1]}{\max}I_{\eta}(\gamma(t))\le c_\eta+(-c'_\eta+2)(\eta-\widetilde{\eta})$.
\end{itemize}
\end{proposition}
\begin{proof}
{\rm (i)} In fact, we can select a sequence $\eta_n\to\eta$ and $\eta_n<\eta$. Since $c_\eta\le {c}_{\eta_n}$, thus $c_\eta\le \lim\limits_{n\to\infty}{c}_{\eta_n}$. If $c_\eta<\lim\limits_{n\to\infty}{c}_{\eta_n}$, let $\delta:=\lim\limits_{n\to\infty}{c}_{\eta_n}-c_\eta>0$. Now by means of the definition of $c_\eta$, we can select a $\overline{\gamma}\in \Gamma$ such that $\underset{t\in [0, 1]}{\max}I_\eta(\overline{\gamma}(t))<c_\eta+\frac{\delta}{2}$. However,
\[
c_{\eta_n}\le \max_{t\in [0, 1]}I_{\eta_n}(\overline{\gamma}(t))<c_\eta
+(\eta-\eta_n)\max_{t\in [0, 1]}\int_\Omega F(\overline{\gamma}(t))\ud x+\frac{\delta}{2}.
\]
This is a contradiction whenever one let $n\to \infty$. Consequently $\lim\limits_{n\to\infty}{c}_{\eta_n}=c_\eta$.

{\rm (ii)} For $\widetilde{\gamma}\in \Gamma$, if \eqref{2509180016} holds and $I_{\eta}(\widetilde{\gamma}(t))\ge c_{\eta}-(\eta-\widetilde{\eta})$, then
\[
\frac{I_{\widetilde{\eta}}(\widetilde{\gamma}(t))-I_{\eta}(\widetilde{\gamma}(t))}{\eta-\widetilde{\eta}}\le \frac{c_{\widetilde{\eta}}+B(\eta-\widetilde{\eta})-c_{\eta}+(\eta-\widetilde{\eta})}{\eta-\widetilde{\eta}}=\frac{c_{\widetilde{\eta}}-c_{\eta}}{\eta-\widetilde{\eta}}+B+1.
\]
Since $c'_{\eta}$ exists, when $\widetilde{\eta}$ is close to 1, we have
\[
\frac{c_{\widetilde{\eta}}-c_{\eta}}{\eta-\widetilde{\eta}}<-c'_\eta+1,
\]
thus
\[
\int_{\Omega}F(\widetilde{\gamma}(t)) \ud x=\frac{I_{\widetilde{\eta}}(\widetilde{\gamma}(t))-I_{\eta}(\widetilde{\gamma}(t))}{\eta-\widetilde{\eta}}\le -c'_{\eta}+B+2.
\]
Hence there exists some $\eta_1\in (\eta_0, \eta)$, such that for $\widetilde{\eta}\in (\eta_1, \eta)$,
\[
\begin{aligned}
\frac12\|\nabla \widetilde{\gamma}(t)\|_{L^2(\Omega)}^2&=I_{\widetilde{\eta}}(\widetilde{\gamma}(t))
+\widetilde{\eta}\int_{\Omega}F(\widetilde{\gamma}(t)) \ud x\\
&\le c_{\widetilde{\eta}}+B(\eta-\widetilde{\eta})+\widetilde{\eta}(-c'_{\eta}+B+2)\\
&\le c_{\eta}+(1+B-c'_\eta)(\eta-\widetilde{\eta})+\widetilde{\eta}(-c'_{\eta}+B+2)\\
&\le C.
\end{aligned}
\]

{\rm (iii)} We can select $\{\gamma\}\subset \Gamma$ satisfying
\begin{equation}\label{2408310308}
\underset{t\in [0, 1]}{\max}I_{\eta}(\gamma(t))\le c_{\widetilde{\eta}}+\eta-\widetilde{\eta}.
\end{equation}
Notice that if we replace $\gamma$ by $|\gamma|$, then \eqref{2408310308} still holds true. Without loss of generality, we can assume ${\rm Im}(\gamma)\subset S^+_\rho$. When $\widetilde{\eta}$ is close to 1,
\[
I_{\eta}(\gamma(t))\le I_{\widetilde{\eta}}(\gamma(t))\le c_{\widetilde{\eta}}+\eta-\widetilde{\eta}\le c_{\eta}+(-c'_{\eta}+2)(\eta-\widetilde{\eta}).
\]
The proof is done.
\end{proof}
To find a bounded (PS) sequence, the following framework will be used, which was proven in \cite[Theorem 4.5 and Remark 4.10]{Ghoussoub1993}, see also \cite[Lemma 2.1]{Pierotti25}.
\begin{proposition}\label{22072013}
Let $X$ be a complete connected Hilbert manifold and $J\in C^1(X,\mathbb{R})$. Let $K\subset X$ be compact and consider a family
\[
\mathcal{E}\subset \{E\subset X: E ~\text{is}~ \text{compact},~K\subset E \},
\]
which is invariant with respect to all deformations leaving $K$ fixed. Assume that
\[
c^*:=\underset{u\in K}{\max}J(u)<c:=\underset{E\in \mathcal{E}}{\inf}\underset{u\in E}{\max}J(u)\in \mathbb{R}.
\]
Then there exist $0<\overline{\delta}<c-c^*$ and $C>0$ such that
\begin{itemize}
\item for every $\delta, \varepsilon\in (0,\overline{\delta})$ with $\varepsilon\le\delta$;
\item for every $E^*\in \mathcal{E}_{\varepsilon}:=\{E\in\mathcal{E}: \max_{u\in E}J(u)\le c+\varepsilon\}$;
\item for every closed set $\mathcal{T}\subset X$ satisfying that $\mathcal{T}$ has nonempty intersection with every element of $\mathcal{E}_{\varepsilon}$ and $\inf_{u\in \mathcal{T}}J(u)\ge c-\delta$,
\end{itemize}
one has that there exists $v^*\in X$ such that
\begin{equation}\label{2509171922}
|J(v^*)-c|+\|J'(v^*)\|+\mbox{dist}(v^*, E^*)+\mbox{dist}(v^*, \mathcal{T})\le C\sqrt{\delta}.
\end{equation}
\end{proposition}
\begin{lemma}\label{2408312351}
If $c_{\eta}$ is differentiable at $\eta$, there exists a bounded sequence $\{u_n\}\subset S_\rho$ such that
\begin{equation}\label{2408311548}
I_{\eta}'|_{S_\rho}(u_n)\to 0,\quad  I_{\eta}(u_n)\to c_\eta \quad \mbox{ and }\quad \|u_n^{-}\|_{L^2(\Omega)}\to 0.
\end{equation}
\end{lemma}
\begin{proof}
By Proposition \ref{2408310238}, we obtain an increasing sequence $\{\eta_n\}$ converging to $\eta$ and a sequence $\{\gamma_n\}$ satisfying the conclusions of Proposition \ref{2408310238}(iii). Let $\varepsilon_n=\eta-\eta_n$, $\delta_n=c_{\eta_n}-c_{\eta}+\eta-\eta_n$. Apply Proposition \ref{22072013} with
\[
 \begin{cases}
 J=I_{\eta_n},\\
  c=c_{\eta_n},\\
  X=S_\rho,\\
   K=\{v_1, v_{\varepsilon_0}\},\\
    E^*=E_n^*:=\{\gamma_n(t): t\in [0,1]\},\\
        \mathcal{E}=\{{\rm Im}(\gamma): {\gamma}\in\Gamma\}, \\
     \mathcal{E}_{\varepsilon}=\mathcal{E}_{\varepsilon_n}:=\{E\in\mathcal{E}: \max_{u\in E}I_{\eta_n}(u)\le c_{\eta_n}+\varepsilon_n\},
 \end{cases}
 \]
 and
\[
\mathcal{T}=\mathcal{T}_n:=\overline{\{u\in S_\rho: I_{\eta}(u)\ge c_\eta-\varepsilon_n\;\;\mbox{with } u\in E,\;\mbox{for some}\; E\in\mathcal{E}_{\varepsilon_n}\}}.
\]
For any $u\in \mathcal{T}_{n}$, there exist some $\widetilde{E}\in \mathcal{E}_{\varepsilon_n}$ with $\widetilde{E}={\rm Im}(\widetilde{\gamma})$ and $t_0\in (0, 1)$ such that $u=\widetilde{\gamma}(t_0)$. Hence
 \[
I_{\eta_n}(u)\ge I_{\eta}(u)\ge c_\eta-\varepsilon_n=c_{\eta_n}-\delta_n.
 \]
 Let $\mathcal{A}:=\cup_{n\ge 1}\mathcal{T}_n$, which is  bounded in $H_0^1(\Omega)$ due to Proposition \ref{2408310238}(ii).

In view of Proposition \ref{22072013},  there exists a sequence $\{u_n\}\subset S_\rho$ such that \eqref{2509171922} hold true. That is,
\begin{equation}\label{2505221203}
| I_{\eta_n}(u_n)-c_{\eta_n}|+\|I_{\eta_n}'|_{S_\rho}(u_n)\|+\text{dist}(u_n, E_n^*)+\mbox{dist}(u_n, \mathcal{T}_n)\le C\sqrt{\delta_n}.
\end{equation}
As a result, $\{u_n\}$ is bounded in $H_0^1(\Omega)$ since $\mathcal{A}$ is  bounded. On the other hand, due to Proposition \ref{2408310238}(i), we get $\delta_n\to 0$ and $\text{dist}(u_n, E_n^*)\to 0$,  combining $E_n^*\subset S_\rho^+$,  it holds that $\|u_n^{-}\|_{L^2(\Omega)}\to 0$ as $n\to\infty$. Finally, the first two terms of \eqref{2408311548} will hold since $\eta_n\to\eta$ and $c_{\eta_n}\to c_\eta$.
\end{proof}

\begin{lemma}\label{2509121227}
Assume that $\Omega$ is a star-shaped domain. Then for any $\delta>0$, there exist $\widehat{\rho}_1>0$, $\eta_1<1$ such that  if $\rho<\widehat{\rho}_1$ and $\eta\in (\eta_1, 1)$,  we have that for any solution  $(\omega, u)$ of \begin{equation}\label{0410031}
\begin{cases}
\begin{aligned}
&-\Delta u+\omega u=\eta\mu |u|^{p-2}u+\eta |u|^{2^*-2}u \quad &&\text { in } \Omega, \\
&u=0  &&\text { on } \partial\Omega, \\
&\|u\|_2^2=\rho>0,
\end{aligned}
\end{cases}
\end{equation}
there holds
\[
I_\eta(u)>
\begin{cases}
\left(\frac{\lambda_1}{N}-\delta\right)\rho\quad&\mbox{if}\;\; 2<p\le 2+\frac{4}{N},\\
\left(\frac{1}{2}-\frac{2}{N(p-2)}\right)\lambda_1\rho\quad&\mbox{if}\;\; 2+\frac{4}{N}<p<2^*.\\
\end{cases}
\]
\end{lemma}
\begin{proof}
This proof will be separated into three cases.

{\it Case $2<p< 2+\frac{4}{N}$}. Since $(\omega, u)$ is a solution of \eqref{0410031}, we easily get
\begin{equation}\label{2509101615}
\|\nabla u\|_2^2+\omega\rho=\eta\int_{\Omega}f(u)u\ud x.
\end{equation}
Since $\Omega$ is a star-shaped domain, we get from Pohozaev identity that
\begin{equation}\label{2509111048}
0< N\left[\eta\int_{\Omega}F(u)\ud x-\frac{\omega\rho}{2}\right]-\frac{N-2}{2}\int_{\Omega}|\nabla u|^2 \ud x,
\end{equation}
which combining \eqref{2509101615} deduces
\begin{equation}\label{24082910400}
\left(\frac{N}{N-2}-1\right)\omega\rho\le \frac{2N}{N-2}\eta\int_{\Omega}F(u)\ud x-\eta\int_{\Omega}f(u)u \ud x=\frac{2^*-p}{p}\mu\eta\|u\|_p^p.
\end{equation}
Using \eqref{2509101615}, \eqref{24082910400}, Gagliardo-Nirenberg inequality and $p>2$, we obtain that for any $\delta>0$, as $\rho$ is small enough, 
\begin{equation}\label{25091401}
\begin{aligned}
I_\eta(u)&=\frac{1}{2}\|\nabla u\|_2^2-\frac{\mu}{p}\eta\|u\|_p^p-\frac{\eta}{2^*}\|u\|_{2^*}^{2^*}\\
&= \frac{1}{N}\|\nabla u\|_2^2-\frac{\omega\rho}{2^*}-\left(\frac{1}{p}-\frac{1}{2^*}\right)\mu\eta\|u\|_p^p\\
&\ge \frac{1}{N}\|\nabla u\|_2^2-\widetilde{C}\mu\eta\|u\|_p^p\\
&\ge \frac{1}{N}\|\nabla u\|_2^2-C\mu\eta\rho^{\frac{p-p\gamma_p}{2}}\|\nabla u\|_2^{p\gamma_p}\\
&\ge \|\nabla u\|_2^{p\gamma_p}\left(\frac{1}{N}\|\nabla u\|_2^{2-p\gamma_p}-C\mu\eta\rho^{\frac{p-p\gamma_p}{2}}\right)\\
&\ge\left(\frac{\lambda_1}{N}-\delta\right)\rho,\\
\end{aligned}
\end{equation}
 where $C$, $\widetilde{C}$ depend only on $N$ and $p$. 
 
 \smallskip
{\it Case $p=2+\frac{4}{N}$}. By a similar argument in \eqref{25091401} and $p\gamma_p=2$, if $\frac{1}{N}>C\mu\eta\rho^{\frac{2}{N}}$, we have
\begin{equation*}
I_\eta(u)\ge \frac{1}{N}\|\nabla u\|_2^2-C\mu\eta\rho^{\frac{2}{N}}\|\nabla u\|_2^{2}
\ge\left(\frac{1}{N}-C\mu\eta\rho^{\frac{2}{N}}\right)\lambda_1\rho.
\end{equation*}
Then, when $\rho$ is small enough, there still holds that $I_\eta(u)>\left(\frac{\lambda_1}{N}-\delta\right)\rho$.

\smallskip
{\it Case $2+\frac{4}{N}<p<2^*$}.
Using \eqref{2509101615} and \eqref{24082910400} again,  eliminating the common term $\omega\rho$, we get
\[
\begin{aligned}
\frac{N(p-2)}{2p}\mu\eta \|u\|_p^p< \|\nabla u\|_2^2-\eta\|u\|_{2^*}^{2^*}.
\end{aligned}
\]
Thus
\begin{equation*}\label{2509110051}
\begin{aligned}
I_\eta(u)&=\frac{1}{2}\|\nabla u\|_2^2-\frac{\mu\eta}{p}\|u\|_p^p-\frac{\eta}{2^*}\|u\|_{2^*}^{2^*}\\
&\ge \left(\frac{1}{2}-\frac{2}{N(p-2)}\right)\|\nabla u\|_2^2+\left(\frac{2}{N(p-2)}-\frac{1}{2^*}\right)\eta\|u\|_{2^*}^{2^*}\\
&>\left(\frac{1}{2}-\frac{2}{N(p-2)}\right)\|\nabla u\|_2^2\\
&\ge \left(\frac{1}{2}-\frac{2}{N(p-2)}\right)\lambda_1\rho.
\end{aligned}
\end{equation*}
The proof is done.
\end{proof}

\begin{lemma}\label{2409010003}
Assume that $N\ge 4$, $2<p<2^*$, or $N=3$ and \eqref{25092601} holds. Then there exist $\rho^*>0$ and $\eta^*\in (0, 1)$ such that  if  $\eta\in (\eta^*, 1)$ and $c_{\eta}$ is differentiable at $\eta$, then for any $\rho<\rho^*$, there exists a  normalized solution $(\omega, u)$ to the equation
\begin{equation}\label{2408311628}
\begin{cases}
-\Delta u+\omega u=\eta f(u),~ u>0 \;\text { in } \Omega, \\
\|u\|^2_{L^2(\Omega)}=\rho,\\
\end{cases}
\end{equation}
and $I_\eta(u)=c_\eta$.
\end{lemma}
\begin{proof}
By Lemma \ref{2408312351}, there exists a sequence $\{(\omega_n, u_n)\}\subset \mathbb{R}\times S_\rho$ such that $\{u_n\}$ is bounded in $H_0^1(\Omega)$, and
\begin{equation}\label{2408311617}
I_\eta(u_n)\to c_\eta\quad\mbox{and}\quad I_{\eta}'(u_n)-\omega_n u_n\to 0\quad \mbox{ in } H^{-1}(\Omega).
\end{equation}
Up to a subsequence, there exists some $u_0\in H_0^1(\Omega)$ such that
\begin{equation*}
\begin{aligned}
{u}_n&\rightharpoonup u_0 ~~\text{ in } H^1_0(\Omega) \text{ and } L^{2^{*}}(\Omega), \\
{u}_n&\rightarrow u_0 ~~\text{ in } L^{r}(\Omega),\;\; r\in [1, 2^*),\\
{u}_n&\rightarrow u_0 ~~{\text{ a.e.}\text{ in }\Omega,}
\end{aligned}
\end{equation*}
which combining \eqref{2408311548} deduces that $u_0\in S_\rho^+$. By \eqref{2408311617}, we have
\[
w_n\rho=\langle I_{\eta}'(u_n), u_n\rangle +o(1).
\]
So $\{\omega_n\}$ is bounded. Hence there exists a subsequence, still denoted by $\{\omega_n\}$, which converges to some $\omega_0$. By ${u}_n\rightharpoonup u_0$ in $H^1_0(\Omega)$ and \eqref{2408311617}, $(\omega_0, u_0)$ is a solution of \eqref{2408311628}.
By maximum principle, $u_0$ is positive. Using Lemma \ref{2509121227}, for any $\delta>0$, provided $\rho$ is small,
\[
I_\eta(u_0)>
\begin{cases}
\left(\frac{\lambda_1(\Omega)}{N}-\delta\right)\rho\quad&\mbox{if}\;\; 2< p\le 2+\frac{4}{N},\\
\left(\frac{1}{2}-\frac{2}{N(p-2)}\right)\lambda_1(\Omega)\rho\quad&\mbox{if}\;\; 2+\frac{4}{N}< p<2^*.\\
\end{cases}
\]

Next, we assume that ${u}_n\rightharpoonup u_0$ in $H^1_0(\Omega)$ is not strong. Proceeding as the proof of \cite[Corollary 2.5]{Pierotti25}, we can obtain
\begin{equation}\label{2511051425}
\underset{n\to\infty}{\lim\inf}I_\eta(u_n)\ge I_\eta(u_0)+\frac{1}{N}{\mathcal{S}}^{\frac{N}{2}}\eta^{1-\frac{N}{2}}.
\end{equation}

When $N\ge 4$, by Lemma \ref{2509121247} and Lemma \ref{2509121227},  this will be a contradiction provided $\rho$ is sufficiently small. 

When $N=3$, if $\frac{14}{3}< p<6$, when $\rho$ is small, arguing as above and using \eqref{25092601}, we can reach a contradiction. If $2< p\le \frac{10}{3}$, we can let $\delta$ be small such that $\frac{\lambda_1(\Omega)}{3}-\delta>\frac{\lambda_1(B_{R_\Omega})}{4}$. Combining Lemmas \ref{2509121247}, \ref{2509121227}, \eqref{25092601} and \eqref{2511051425}, when $\rho$ is small, we still get a contradiction.
\end{proof}
\begin{proposition}\label{2408302031}
For $\rho<\rho^*$, there exists a sequence $\{(\eta_n, u_n)\}\subset (0, 1)\times S_\rho^+$ such that
\begin{itemize}
\item[\rm (i)] $\{\eta_n\}$ is increasing and tends to 1;
\item[\rm (ii)] $I_{\eta_n}'|_{S_\rho}(u_n)=0$;
\item[\rm (iii)] $I_{\eta_n}(u_n)=c_{\eta_n}$. In particular, $\lim\limits_{n\to\infty}{c}_{\eta_n}=c_{1}$,
\end{itemize}
where $\rho^*$ is given in Lemma \ref{2409010003}.
\end{proposition}
\begin{proof}
As we mentioned before, $c_{\eta}$ is non-increasing in $\eta$, so $c_{\eta}$ is almost everywhere differentiable with respect to $\eta$. Hence, there exists an increasing sequence $\{\eta_n\}$ converging to 1, and such that $c_\eta$ is differentiable at every $\eta_n$. By Lemma \ref{2409010003}, there exists  $u_n$ satisfying $I_{\eta_n}'|_{S_\rho}(u_n)=0$ and $I_{\eta_n}(u_n)=c_{\eta_n}$. Finally, by Proposition \ref{2408310238}(i), we obtain $c_{\eta_n}\to c_1$.
\end{proof}

\subsection{Proof of Theorem \ref{T041805}}

Let $\{(\eta_n, u_n)\}\subset (0, 1)\times S_\rho^+$ be given in proposition \ref{2408302031}. From  $I_{\eta_n}'|_{S_\rho}(u_n)=0$, we get that there exists a sequence of multipliers $\{\omega_n\}\subset \mathbb{R}$ such that
\begin{equation}\label{2505231850}
-\Delta u_n+\omega_n u_n=\eta_nf(u_n) \quad\text { in } \Omega.
\end{equation}
The equation above can deduce the Pohozaev identity \eqref{2408291040} 
and the Nehari-type identity
\begin{equation}\label{2408291041}
\int_{\Omega}|\nabla u_n|^2 \ud x+\omega_n\rho=\eta_n\int_{\Omega}f(u_n)u_n \ud x.
\end{equation}

 Note that multiplying \eqref{2505231850} by $\varphi_{1}$ and integrating over $\Omega$, we deduce that
\begin{equation*}
\begin{aligned}
(\lambda_1(\Omega)+\omega_n)\int_{\Omega} u_n\varphi_{1} \ud x
=\eta_n\int_{\Omega}f(u_n)\varphi_{1} \ud x>0.
\end{aligned}
\end{equation*}
Thus
\begin{equation}\label{2510092110}
\omega_n>-\lambda_1(\Omega).
\end{equation}

\begin{proposition}\label{2505231848}
 Assume that $\Omega$ is a star-shaped domain. Let $N\geq 3$, $\mu\ge 0$, $\alpha>\frac{\|\nabla U_1\|_2^2}{\|U_1\|_2^2}$, and $\{(\eta_n, u_n)\}\subset (0, 1)\times S_\rho^+$ be given by Proposition \ref{2408302031}.
Then $\{u_n\}$ is bounded in $H_0^1(\Omega)$.
\end{proposition}
\begin{proof}
We assume by contradiction that up to a subsequence, $\|\nabla u_n\|_2\to \infty$. Since $\Omega$ is a star-shaped domain, we get from \eqref{2408291040} that
\begin{equation*}
0\le N\left[\eta_n\int_{\Omega}F(u_n)\ud x-\frac{\omega_n\rho}{2}\right]-\frac{N-2}{2}\int_{\Omega}|\nabla u_n|^2 \ud x,
\end{equation*}
which combining \eqref{2408291041} deduces
\begin{equation}\label{2408291042}
\left(\frac{N}{N-2}-1\right)\omega_n\rho\le \frac{2N}{N-2}\eta_n\int_{\Omega}F(u_n)\ud x-\eta_n\int_{\Omega}f(u_n)u_n \ud x.
\end{equation}
Suppose that $\{\omega_n\}$ is unbounded, according to \eqref{2510092110}, up to a subsequence, we can assume $\omega_n\to\infty$.  It follows from \eqref{2408291042} that
\[
\omega_n\rho\le \frac{(2^*-p)(N-2)}{2p}\mu\eta_n\|u_n\|_{p}^{p},
\]
which together with \eqref{2408291041} and H\"older inequality, deduces that
\begin{equation}\label{2408291043}
\|\nabla u_n\|_2^2=\|u_n\|_{2^*}^{2^*}+o(1)\|u_n\|_{2^*}^{2^*}.
\end{equation}
On the other hand, by \eqref{2408291043} and $I_{\eta_n}(u_n)\to c_1$, we have
\[
\begin{aligned}
c_1&=I_{\eta_n}(u_n)+o(1)\\
&=\frac{1}{2}\|\nabla u_n\|_2^2-\frac{1}{2^*}\|\nabla u_n\|_2^2+o(1)\|\nabla u_n\|_2^2+o(1)\\
&\to \infty.
\end{aligned}
\]
This is a contradiction. So $\{\omega_n\}$ is bounded. 

By \eqref{2408291041}, we have
\[
\begin{aligned}
\frac{1}{2}\|\nabla u_n\|^2_{2}=&I_{\eta_n}(u_n)
+\eta_n\int_{\Omega}F(u_n) \ud x\\
=&I_{\eta_n}(u_n)+\frac{\mu\eta_n}{p}\|u_n\|_p^p+\frac{\eta_n}{2^*}\|u_n\|_{2^*}^{2^*}\\
=&I_{\eta_n}(u_n)+\frac{1}{p}\left(\|\nabla u_n\|^2_{2}+\omega_n\rho
-\eta_n\|u_n\|_{2^*}^{2^*}\right)
+\frac{\eta_n}{2^*}\|u_n\|_{2^*}^{2^*},
\end{aligned}
\]
thus
\[
\begin{aligned}
\left(\frac{1}{2}-\frac{1}{p}\right)\|\nabla u_n\|^2_{2}
=I_{\eta_n}(u_n)+\frac{1}{p}\omega_n\rho
+\eta_n\left(\frac{1}{2^*}-\frac{1}{p}\right)\|u_n\|_{2^*}^{2^*},
\end{aligned}
\]
which is impossible as $n\rightarrow\infty$. So $\{u_n\}$ is also bounded in $H_0^1(\Omega)$.
\end{proof}
\noindent
{\bf Proof of Theorem \ref{T041805} completed.} The proof will be split into Sobolev critical case and subcritical case.

{\it Case $2<p<q=2^*$. } Using Proposition \ref{2505231848}, we get a sequence $\{u_n\}$, which is bounded in $H_0^1(\Omega)$, and that are the solutions of \eqref{2505231850}. By \eqref{2408291041}, $\{\omega_n\}$ is also bounded. Therefore, up to a subsequence, there is $(\omega, u)\in  \mathbb{R}\times S_\rho $ such that $u_n\rightharpoonup u$ in $H_0^1(\Omega)$, $u_n(x)\to u(x)$ a.e. in $\Omega$ and $\omega_n\to \omega$.
Assume that ${u}_n\rightharpoonup u$ in $H^1_0(\Omega)$ is not strong. Similar to the proof of Lemma \ref{2409010003}, we will get
\[
\underset{n\to\infty}{\lim\inf}I_{\eta_n}(u_n)\ge I(u)+\frac{1}{N}S^{\frac{N}{2}}.
\]
However, by Lemma \ref{2509121247} and Lemma \ref{2509121227} with $\eta=1$, this will be a contradiction provided $\rho$ is sufficiently small.
By a convergence argument and maximum principle, we can conclude that $(\omega, u)$ is a solution of \eqref{041003} satisfying $I(u)=c_1$. Since $c_1\ge \underset{v\in \partial A_\alpha}{\inf}I(v)>\underset{v\in A_\alpha}{\inf}I(v)$, so $u$ is the second solution of \eqref{041003}.

\medskip
{\it Case $2<p<2+\frac{4}{N}<q<2^*$.} For the Sobolev subcritical situation, we consider the energy functionals
\[
I_{\eta}(u)=\frac{1}{2}\int_{\Omega}|\nabla u|^2 \ud x
-\eta\int_{\Omega}\left(\frac{\mu}{p} |u|^p+\frac1{q}|u|^{q}\right)  \ud x,\quad u\in S_\rho\;\;\mbox{and}\;\; \eta>0.
\] 
For the mass supercritical nonlinearities at infinity, it is easy to establish mountain pass structure. That is, $2+\frac{4}{N}< q$ is enough to guarantee mountain pass geometry. Since the embeddings $H_0^1(\Omega)\subset L^p(\Omega)$ and $H_0^1(\Omega)\subset L^q(\Omega)$ are compact, we will not need the bubbling estimate. Still we use the monotonicity trick. As Proposition \ref{2408302031}, with the same notations $c_1$ and $c_\eta$ as before,  there exists a sequence $\{(\eta_n, u_n)\}\subset (0, 1)\times S_\rho^+$ such that
\begin{itemize}
\item[\rm (i)] $\{\eta_n\}$ is increasing and tends to 1;
\item[\rm (ii)] $I_{\eta_n}'|_{S_\rho}(u_n)=0$;
\item[\rm (iii)] $I_{\eta_n}(u_n)=c_{\eta_n}$, and $\lim\limits_{n\to\infty}{c}_{\eta_n}=c_{1}$.
\end{itemize}

For the sake of the boundedness of $\{u_n\}$, we need the constraint $2<p<2+\frac{4}{N}$. Indeed,  we suppose $\|\nabla u_n\|_2\to\infty$. Similar to \eqref{2408291040} and \eqref{2408291041}, we can get that
\begin{equation*}
\eta_n\|u_n\|^q_q=\frac{q}{N(q-2)}\|\nabla u_n\|^2_2
-\frac{q}{N(q-2)}\int_{\partial \Omega}|\nabla u_n|^2\sigma\cdot \nu \ud\sigma
-\frac{(p-2)q}{(q-2)p}\eta_n\mu\|u_n\|^p_p.
\end{equation*}
By \eqref{050502} and $p<2+\frac{4}{N}$, and since $\Omega$ is a star-shaped domain,
\begin{equation*}
\begin{aligned}
I_{\eta_n}(u_n)=&\frac{1}{2}\|\nabla u_n\|^2_2
-\eta_n\left(\frac{\mu}{p}\|u_n\|^{p}_p
+\frac{1}{q}\|u_n\|^{q}_q \right)\\
=&\frac{1}{N(q-2)}\|\nabla u_n\|^2_2
-\frac{q-p}{(q-2)p}\eta_n\mu \|u_n\|^{p}_p
+\frac{1}{N(q-2)}\int_{\partial \Omega}|\nabla u_n|^2\sigma\cdot \nu \ud\sigma\\
\geq&\frac{1}{N(q-2)}\|\nabla u_n\|^2_2
-\frac{q-p}{(q-2)p}\eta_n\mu  \cdot\mathcal{C}^p_{N,p}\|\nabla u_n\|^{p\gamma_{p}}_2
\|u_n\|^{p(1-\gamma_{p})}_2\\
\rightarrow&+\infty
\end{aligned}
\end{equation*}
as $n\to\infty$, which contradicts $I_{\eta_n}(u_n)\to c_1$.

\section{Non-existence for normalized solutions with large mass in a bounded convex domain}\label{S_4}
In this section, we consider the problem \eqref{041003} with $N\ge 3$ and $q\ge 2^*$ in a convex domain. If $u$ is a solution of \eqref{041003}, using moving-plane method, we will see that $u$ is monotonic along some directions around the boundary. This could deduce some nonexistence result of \eqref{041003} when the mass is large. First, we shall estimate the frequency $\omega$ of \eqref{041003} in a star-shaped domain.

\begin{lemma}\label{2410040137}
Assume that $\Omega$ is a smooth star-shaped domain with respect to 0, $\mu>0$, $1<p\le 2$ and $q\ge 2^*$. If $(\omega, u)$ is a solution of \eqref{041003}, then there holds that
\begin{equation}\label{2505261132}
-\lambda_1(\Omega)< \omega\le N\left(\frac{1}{p}-\frac{1}{q}\right)\mu|\Omega|^{\frac{2-p}{2}}\rho^{\frac{p-2}{2}}.
\end{equation}
\end{lemma}
\begin{proof}
As \eqref{2510092110}, one can get that $-\lambda_1(\Omega)<\omega$. Next, by the Pohozaev identity
\begin{equation}\label{2505261102}
\frac12\int_{\partial \Omega}|\nabla u|^2\sigma\cdot \nu \ud\sigma=N\left[\frac{\mu\|u\|_p^p}{p}+\frac{\|u\|_{q}^{q}}{q}-\frac{\omega\rho}{2}\right]-\frac{N-2}{2}\int_{\Omega}|\nabla u|^2 \ud x
\end{equation}
and the Nehari-type identity
\begin{equation}\label{2505261103}
\int_{\Omega}|\nabla u|^2 \ud x+\omega\rho=\mu\|u\|_{p}^p+\|u\|_{q}^{q},
\end{equation}
we deduce  that
\[
\begin{aligned}
\frac12\int_{\partial \Omega}|\nabla u|^2\sigma\cdot \nu \ud\sigma+\omega\rho&=N\left(\frac{1}{p}-\frac{1}{q}\right)\mu\|u\|_p^p+\left(\frac{N}{q}-\frac{N-2}{2}\right)\int_{\Omega}|\nabla u|^2 \ud x\\
&\le N\left(\frac{1}{p}-\frac{1}{q}\right)\mu|\Omega|^{\frac{2-p}{2}}\|u\|_2^p\\
&=N\left(\frac{1}{p}-\frac{1}{q}\right)\mu|\Omega|^{\frac{2-p}{2}}\rho^{\frac{p}{2}},
\end{aligned}
\]
which implies the right side of \eqref{2505261132}. The proof is completed.
\end{proof}

Recall the classical moving-plane method as in \cite{GNN1979}. If all sectional curvatures at every point of $\partial\Omega$ are positive, one can carry out the steps of moving-plane method near the boundary. As mentioned in \cite{dLN1982}, the above arguments also work in general convex domains. Accordingly, the positive solutions of  \eqref{041003} along some directions near the boundary are monotonic. Furthermore, one can obtain the following lemma.  The interested readers can refer to the proof of \cite[Theorem 1.1]{dLN1982}.
\begin{lemma}\label{2506192013}
Let  $\Omega$ be a bounded convex domain with smooth boundary. There exist $\beta, t_0>0$ (depending only on $\Omega$) such that any positive solution to \eqref{041003} satisfies
\begin{equation}\label{2506192108}
\begin{cases}
u(x-t\nu) \mbox{ is non-decreasing for  $t\in [0, t_0)$ where $\nu\in\mathbb{R}^N$ such that} \\
\mbox{$|\nu|=1$ and $\nu\cdot n(x)\ge \beta$ for $x\in \partial\Omega$, $n(x)$ is the outer normal direction at $x$.}
\end{cases}
\end{equation}
As a corollary,  there exist some $r>0$ and $C=C(\Omega, N, p, q, r)>0$ such that any classical positive solution to \eqref{041003} satisfies
\[
\int_{\Omega_r}u^2 \ud x\le C\int_{\Omega_{2r}\backslash \Omega_r} u^2 \ud x,
\]
where $\Omega_{r}:=\{x\in\Omega: {\rm dist}(x, \partial\Omega)<r\}$.
%moreover for any $x\in\Omega_r$, there exists a set $I_x\subset \Omega\backslash \Omega_{r}$ which is a piece of a cone with vertex $x$, such that $|I_x|\ge \gamma$ and
%\[
%u(x)\le u(y),\quad\quad \forall\, y\in I_x.
%\]
\end{lemma}

\medskip
\noindent$\mathbf{Proof~of~Theorem~\ref{T051401}}$.  Assume that  $(\omega, u)\subset \mathbb{R}\times H^1_0(\Omega)$ is a solution of \eqref{041003}. Using the first eigenfunction $\varphi_1>0$ as a test function to multiply \eqref{041003}, we get
\begin{equation}\label{2509241809}
\begin{aligned}
\lambda_1+\omega\ge\frac{\int_{\Omega}u^{q-1}\varphi_1 \ud x}
{\int_{\Omega}u\varphi_1 \ud x}.
\end{aligned}
\end{equation}
By Lemma \ref{2506192013}, there exist some $r>0$, $C_1=C_1(\Omega, N, p, q, r)>0$  such that
\begin{equation}\label{2509241810}
\begin{aligned}
\rho>\|u\|^2_{L^2(\Omega\backslash\Omega_r)}\geq C_1\rho.\\
\end{aligned}
\end{equation}
By H\"older inequality, we have
\begin{equation}\label{25092418100}
\begin{aligned}
\int_{\Omega}u\varphi_1 \ud x\leq \|\varphi_1\|_2\rho^{\frac12}.
\end{aligned}
\end{equation}

\smallskip
Let us denote $m:=\inf\limits_{x\in \Omega\backslash\Omega_r}\varphi_1(x)>0$. By \eqref{2509241809}, \eqref{2509241810}, H\"older inequality and $q-1\ge 2$, we get that
\begin{equation}\label{2505261543}
\begin{aligned}
\int_{\Omega}u^{q-1}\varphi_1 \ud x\ge \int_{\Omega\backslash\Omega_r}u^{q-1}\varphi_1 \ud x\ge m\int_{\Omega\backslash\Omega_r}u^{q-1} \ud x\ge mC_1^{\frac{q-1}{2}}|\Omega\backslash\Omega_r|^{\frac{3-q}{2}}\rho^{\frac{q-1}{2}}.
\end{aligned}
\end{equation}
Using \eqref{2509241809}-\eqref{2505261543}, we obtain that
\begin{equation*}
\begin{aligned}
\lambda_1+\omega\ge  \frac{mC_1^{\frac{q-1}{2}}|\Omega\backslash\Omega_r|^{\frac{3-q}{2}}\rho^{\frac{q-1}{2}}}
{\|\varphi_1\|_2\rho^{\frac12}}.
\end{aligned}
\end{equation*}
 By Lemma \ref{2410040137}, $\rho$ is bounded from above provided $1<p\le 2$. So the proof is done.

\section{Dichotomy result for Br\'{e}zis-Nirenberg problem in a ball}\label{S_5}

In this section, we consider the Br\'{e}zis-Nirenberg problem in a ball subject to prescribed mass, namely
\begin{equation}\label{060908}
\begin{cases}
-\Delta u+\omega u=u^{2^{*}-1},~ u>0 \quad\text { in } B_1(0), \\
\|u\|^2_{L^2(B_1(0))}=\rho,\\
u\in H^1_0(B_1(0)).
\end{cases}
\end{equation}
 We learn from \cite{Zhang92} that for $N\ge 4$ and $\lambda\in \left(0, \lambda_1(B_1(0))\right)$, or $N=3$ and $\lambda\in \left(\frac{\lambda_1(B_1(0))}{4},\lambda_1(B_1(0))\right)$, the solution of \eqref{060909} is unique and radially symmetric. We denote this unique solution by $u_\lambda$. Comparing \eqref{060908} with \eqref{060909}, we only don't know the $L^2$-norm of $u_\lambda$, which motivates us to explore the tendency of $\|u_\lambda\|_2$ as $\lambda$ varies in either $(0,\lambda_1(B_1(0)))$ when $N>3$ or $\left(\frac{\lambda_1(B_1(0))}{4},\lambda_1(B_1(0))\right)$ when $N=3$. The first question is whether $\|u_\lambda\|_2$ is continuous with respect to the parameter $\lambda$. Since the solution of \eqref{060909} is radial and unique, \eqref{060909} is equivalent to the following ordinary differential equation 
 \begin{equation*}
\begin{cases}
w''+\frac{N-1}{r}w'+\lambda w+w^{2^{*}-1}=0,\\
w'(0)=0, w(1)=0.
\end{cases}
\end{equation*}
 By the continuity theorem of ordinary differential equation on initial values and parameter, it can be seen that
\begin{lemma}\label{060920}
Suppose that $N\ge 3$ and $u_\lambda$ is the unique solution to \eqref{060909}. Then when $N>3$, $u_\lambda$ is continuous in $L^{\infty}(B_1(0))$ with respect to  $\lambda\in (0,\lambda_1(B_1(0)))$; when $N=3$, $u_\lambda$ is continuous in $L^{\infty}(B_1(0))$ with respect to  $\lambda\in \left(\frac{\lambda_1(B_1(0))}{4},\lambda_1(B_1(0))\right)$.
\end{lemma}
The another question is: what's the tendency of $\|u_\lambda\|_2$ as $\lambda\to 0$  for $N>3$, and as $\lambda\to\frac{\lambda_1(B_1(0))}{4}$ for $N=3$, and as $\lambda\to \lambda_1(B_1(0))$ for $N\ge 3$? It is answered by our next lemma.
\begin{lemma}\label{060921}
Suppose that $N\ge 3$ and $u_\lambda$ is the unique solution to \eqref{060909}. Then when $N>3$, $\|u_\lambda\|_2\to 0$ as $\lambda\to 0$ and $\lambda\to \lambda_1(B_1(0))$, and when $N=3$, $\|u_\lambda\|_2\to 0$ as $\lambda\to {\frac{ \lambda_1(B_1(0))}{4}}$ and $\lambda\to \lambda_1(B_1(0))$.
\end{lemma}

Concerning the proof of Lemma \ref{060921}, some related results were established in \cite{Rey1990, Zhang92}. For completeness, we give it here.

\begin{proof}[\bf Proof of Lemma \ref{060921}]
We consider the minimization problem
\[
\theta_\varepsilon:=\inf\left\{J_{\varepsilon}(u)=\|\nabla u\|_2^2-\varepsilon \|u\|_{2^*}^2,\;\; u\in H_0^1(B_1(0)),\;\|u\|_2=1\right\}.
\]
By \cite[Lemma 4.1]{Zhang92},  for $\varepsilon\in (0, \mathcal{S})$, $\theta_{\varepsilon}$ can be attained, where $\mathcal{S}$ is the best Sobolev constant. By \cite[Lemma 4.2]{Zhang92}, we have that $\theta_\varepsilon$ is continuous and decreasing with respect to the parameter $\varepsilon$, and
\[
\begin{aligned}
&\theta_{\varepsilon}\to \lambda_1(B_1(0))\quad \mbox{as } \varepsilon\to 0,\\
&\theta_{\varepsilon}\to
\begin{cases}
0,\quad &N\ge 4,\\
\frac{\lambda_1(B_1(0))}{4},\quad &N=3,
\end{cases}
\quad \mbox{as }\varepsilon\to \mathcal{S}.
\end{aligned}
\]
The minimizer of $\theta_\varepsilon$ can be denoted by $v_\varepsilon$, which satisfies
\begin{equation}\label{2410050043}
\begin{cases}
-\Delta v_\varepsilon= \theta_\varepsilon v_\varepsilon+\varepsilon \|v_\varepsilon\|_{2^*}^{2-2^*}v_\varepsilon^{2^*-1},\;\;\; v_\varepsilon>0\;\; \mbox{in } B_1(0),\\
v_\varepsilon=0\;\; \mbox{on } \partial B_1(0).
\end{cases}
\end{equation}
By the uniqueness,
\begin{equation}\label{2410050047}
u_\lambda=\varepsilon^{\frac{1}{2^*-2}}\frac{v_\varepsilon}{\|v_\varepsilon\|_{2^*}},\quad \mbox{with }\lambda=\theta_\varepsilon.
\end{equation}

 We claim that $\{u_\lambda\}$ is bounded in $H_0^1(B_1(0))$ as $\lambda\to \lambda_1$. In fact, it follows from \eqref{2410050043} and \eqref{2410050047} that
\begin{equation}\label{2410050058}
\|\nabla v_\varepsilon\|_2^2=\theta_\varepsilon +\varepsilon \|v_\varepsilon\|_{2^*}^2,
\end{equation}
and
\[
\|\nabla u_{\theta_\varepsilon}\|_2^2=\varepsilon^{\frac{2}{2^*-2}}\|v_\varepsilon\|^{-2}_{2^*}\|\nabla v_\varepsilon\|_2^2.
\]
For $\varepsilon<\frac{\mathcal{S}}{2}$, we conclude from \eqref{2410050058} and Sobolev inequality that
\[
\frac{1}{2}\|\nabla v_\varepsilon\|_{2}^2\le \left(1-\frac{\varepsilon}{\mathcal{S}}\right)\|\nabla v_\varepsilon\|_{2}^2\le \|\nabla v_\varepsilon\|_2^2-\varepsilon \|v_\varepsilon\|_{2^*}^2= \theta_\varepsilon<\lambda_1(B_1(0)).
\]
As a consequence, $\{v_\varepsilon\}$ is bounded in $H_0^1(B_1(0))$ as $\varepsilon\to 0$. By \eqref{2410050043}, $\|v_\varepsilon\|_2=1$ and $\theta_\varepsilon\to \lambda_1$, $v_\varepsilon$ must converge to the first eigenfunction $\varphi_1$ as $\varepsilon\to 0$. It follows from \eqref{2410050047} that $u_\lambda\to 0$ in $H_0^1(B_1(0))$.
\smallskip

In the sequel, we shall show that $\{u_\lambda\}$ is bounded in $H_0^1(B_1(0))$ as either $\lambda\to 0$ for $N\ge 4$, or $\lambda\to \frac{\lambda_1}{4}$ for $N=3$. In the following, we only consider $N\ge 4$, the case $N=3$ is completely a same procedure.

For $\varepsilon\in (\frac{\mathcal{S}}{2}, \mathcal{S})$, if $\{v_\varepsilon\}$ is unbounded in $H_0^1(B_1(0))$, we can assume that as $\varepsilon\to \mathcal{S}$,
\[
\|\nabla v_\varepsilon\|_2\to\infty.
\]
Form \eqref{2410050058} and $\theta_{\varepsilon}\to 0$ as $\varepsilon\to\mathcal{S}$, we have
\begin{equation}\label{2505250035}
\frac{\|\nabla v_{\varepsilon}\|_2^2}{\|v_{\varepsilon}\|_{2^*}^2}\to \mathcal{S}.
\end{equation}
It follows from \eqref{2410050047} and \eqref{2505250035} that
$$\|\nabla u_{\theta_\varepsilon}\|_2=\varepsilon^{\frac{1}{2^*-2}}\frac{\|\nabla v_{\varepsilon}\|_2}{\|v_{\varepsilon}\|_{2^*}}\to \mathcal{S}^{\frac{2^*}{2(2^*-2)}}$$
 as $\varepsilon\to\mathcal{S}$. So $\{u_\lambda\}$ is bounded in $H_0^1(B_1(0))$ as $\lambda\to 0$.

 \smallskip
 Next, we assume $\{v_\varepsilon\}$ is bounded in $H_0^1(B_1(0))$ as $\varepsilon\to \mathcal{S}$.  Since $\|v_\varepsilon\|_2=1$, $\{v_\varepsilon\}$ has a positive lower bound in $L^{2^*}(B_1(0))$, which together with \eqref{2410050047} implies that $\{u_\lambda\}$ is bounded in $H_0^1(B_1(0))$.

 \smallskip

Up to a subsequence, there exists some $u\in H_0^1(B_1(0))$ such
 that $u_\lambda \rightharpoonup u$ as $\lambda\to 0$. Therefore $u$ is a solution of
 \[
 -\Delta u=u^{2^*-1}
 \]
thus $u=0$. Finally, we have $u_\lambda\to 0$ in $L^2(B_1(0))$ as $\lambda\to 0$.
\end{proof}

\noindent{\bf Proof of Theorem \ref{2505261703} completed.} From Lemma \ref{060920} and Lemma \ref{060921},  it is easily known that there are two solutions to (\ref{060908}) as $\rho$ is small, and no solutions as $\rho$ is large, and Theorem \ref{2505261703} can be proved.

%\subsection{A lower bound  for $\rho^*$}
%
%In this subsection, we will give a lower bound for $\rho^*$, where $\rho^*$ is given in Theorem \ref{2505261703}. First, 

\end{document}